\documentclass[12pt]{amsart}
\usepackage{amssymb}
\usepackage{amsmath,amscd}

\let\frak\mathfrak
\let\Bbb\mathbb

\def\>{\relax\ifmmode\mskip.666667\thinmuskip\relax\else\kern.111111em\fi}
\def\<{\relax\ifmmode\mskip-.333333\thinmuskip\relax\else\kern-.0555556em\fi}
\def\vsk#1>{\vskip#1\baselineskip}
\def\vv#1>{\vadjust{\vsk#1>}\ignorespaces}
\def\vvn#1>{\vadjust{\nobreak\vsk#1>\nobreak}\ignorespaces}

\let\Medskip\medskip
\def\medskip{\par\Medskip}
\let\Bigskip\bigskip
\def\bigskip{\par\Bigskip}

\let\Maketitle\maketitle
\def\maketitle{\hrule height0pt\vskip-\baselineskip
\Maketitle\thispagestyle{empty}\let\maketitle\empty}

\newtheorem{thm}{Theorem}[section]
\newtheorem{cor}[thm]{Corollary}
\newtheorem{lem}[thm]{Lemma}
\newtheorem{prop}[thm]{Proposition}

\newtheorem{defn}[thm]{Definition}

\numberwithin{equation}{section}

\theoremstyle{definition}

\let\mc\mathcal
\let\nc\newcommand

\nc{\on}{\operatorname}
\nc{\Z}{{\mathbb Z}}
\nc{\C}{{\mathbb C}}
\nc{\N}{{\mathbb N}}
\nc{\pone}{{\mathbb C}{\mathbb P}^1}
\nc{\arr}{\rightarrow}
\nc{\larr}{\longrightarrow}
\nc{\al}{\alpha}
\nc{\W}{{\mc W}}
\nc{\la}{\lambda}
\nc{\su}{\widehat{{\mathfrak sl}}_2}
\nc{\g}{{\mathfrak g}}
\nc{\h}{{\mathfrak h}}
\nc{\m}{{\mathfrak m}}
\nc{\n}{{\mathfrak n}}
\nc{\Gm}{\Gamma}
\nc{\La}{\Lambda}
\nc{\gl}{\widehat{\mathfrak{gl}_2}}
\nc{\bi}{\bibitem}
\nc{\om}{\omega}
\nc{\Res}{\on{Res}}
\nc{\gm}{\gamma}
\nc{\Om}{\Omega}

\def\Res{\on{Res}}

\def\B{{\mc B}}

\def\V{{\mc V}}

\let\Dl\Delta
\let\eps\varepsilon

\let\geq\geqslant

\let\leq\leqslant

\nc{\gln}{\mathfrak{gl}_N}
\nc{\sln}{\mathfrak{sl}_N}

\def\beq{\begin{equation}}
\def\eeq{\end{equation}}
\def\be{\begin{equation*}}
\def\ee{\end{equation*}}

\nc{\bean}{\begin{eqnarray}}
\nc{\eean}{\end{eqnarray}}
\nc{\bea}{\begin{eqnarray*}}
\nc{\eea}{\end{eqnarray*}}
\nc{\bs}{\boldsymbol}
\nc{\Ref}[1]{{\rm(\ref{#1})}}
\nc{\glN}{\mathfrak{gl}_N}
\nc{\glNt}{\mathfrak{gl}_N[t]}
\nc{\s}{sing}
\nc{\R}{\Bbb R}

\nc{\Oml}{{\Om_{\bs\la}}}
\nc{\OmLb}{{\Om_{\bs\La,\bs\la,\bs b}}}
\nc{\Ol}{{\mc O_{\bs\la}}}
\nc{\OLb}{{\mc O_{\bs\La,\bs\la,\bs b}}}
\nc{\VSl}{{(\V^S)_{\bs\la}}}
\nc{\Bl}{{\B_{\bs\la}}}
\nc{\Ml}{{\mc M_{\bs\la}}}
\nc{\Mlb}{{\mc M_{\bs\La,\bs\la,\bs b}}}
\nc{\Blb}{{\B_{\bs\La,\bs\la,\bs b}}}
\nc{\Omn}{{\Omega_{\bs n,\bs b,\bs K}}}
\nc{\Omlb}{{\bar\Om_{\bs\la}}}
\nc{\ep}{\epsilon}

\nc{\Dlb}{\Dl_{\bs\La,\bs\la,\bs b,\bs K}}
\nc{\Bb}{{\bf b}}
\nc{\glt}{{\frak{gl}_2}}
\nc{\A}{{\mc A}}
\nc{\slt}{{\frak{sl}_2}}
\nc{\Ma}{{\mc M_{\bs a}}}
\nc{\Mal}{{\mc M_{\bs\la,\bs a}}}
\nc{\Malp}{{\mc M_{\phi,\bs\la,\bs a}}}
\nc{\Vl}{{\V^S_{\bs\la}}}
\nc{\Bal}{{\B_{\bs\la,\bs a}}}
\nc{\Ola}{{\mc O_{\bs\la,\bs a}}}
\nc{\Bv}{{\mc B_{\V^S}}}
\nc{\Bvz}{{\mc B^0_{\V^S}}}

\nc{\sing}{{\rm Sing\,}}
\nc{\Uglt}{U(\glt)}
\nc{\Olo}{{\mc O^0_{\bs\la}}}

\nc{\hyp}{{\vphantom{F}_2 F_1}}

\begin{document}

\title[Norms of eigenfunctions to trigonometric KZB operators]
{Norms of eigenfunctions to trigonometric KZB operators}

\author[E.\, Jensen and  A.\,Varchenko]
{E.\, Jensen and  A.\,Varchenko$\>^\diamond$}

\maketitle

\begin{center}
\vsk-.2>
{\it Department of Mathematics, University of North Carolina
at Chapel Hill\\ Chapel Hill, NC 27599-3250, USA\/}
\end{center}

{\let\thefootnote\relax
\footnotetext{\vsk-.8>\noindent
$^\diamond$\,Supported in part by NSF grants DMS-0555327 and DMS-1101508}}

\medskip
\begin{abstract}
Let $\g$ be a simple Lie algebra and $V[0]=V_1\otimes \cdots \otimes V_n[0]$ the
zero weight subspace of a tensor product of $\g$-modules. The trigonometric KZB
operators are commuting differential operators acting on $V[0]$-valued functions
on the Cartan subalgebra of $\g$. Meromorphic eigenfunctions to the operators are constructed
by the Bethe ansatz. We introduce a scalar product on a suitable space of functions such that the
operators become symmetric, and the square of the norm of a Bethe eigenfunction
equals the Hessian of the master function at the corresponding critical
point.
\end{abstract}

\section{Introduction}

We study three systems of commuting linear operators associated to a simple Lie algebra
$\g$ and the tensor product
$V=V_1\otimes \cdots \otimes V_n$ of finite dimensional representations of $\g$.
The first system is the collection of the rational Gaudin operators acting
on the space of singular vectors of
$M_{\xi -\rho} \otimes V$ of weight $\xi - \rho$, where $M_{\xi-\rho}$ is the
Verma module of highest weight $\xi-\rho$.
The second system is the collection of the trigonometric
Gaudin operators with parameter $\xi$ acting on the zero weight subspace $V[0]\subset V$.
The third system is the collection
the trigonometric KZB differential operators acting on a particular
space
$E(\xi)$ of $V[0]$-valued functions on the Cartan subalgebra $\h \subset \g$.
The three systems are isomorphic.

Each system has a Bethe ansatz construction of eigenvectors.
The Bethe ansatz consists of a scalar master function of
auxiliary  variables $t = (t_1, \dots, t_k)$ and a weight
function depending on $t$. For a critical point $t_{cr}$ of
the master function, the value of the weight function at $t_{cr}$
is an eigenvector.

A large body of the previous work focuses on the norms of such Bethe eigenvectors
in the case of the rational Gaudin operators with respect to
the Shapovalov form \cite{MV}, \cite{V3}.
In that case, if $t_{cr}$ is an isolated
critical point of the master function, then the norm of the eigenvector
corresponding to $t_{cr}$
equals the Hessian of the master function at $t_{cr}$.
In this paper, we extend this result to the eigenvectors of the
other two systems, see Section \ref{sec Bethe}.

In Section \ref{sect:weylgroup}, we describe the Weyl group action
on eigenfunctions to the trigonometric KZB operators, and show that
the scalar product on these functions is Weyl invariant. In section
\ref{sect:jack}, we recall the construction of Jack polynomials from
antisymmetrized eigenfunctions to the trigonometric KZB operators.
We apply our results to relate the usual norm of Jack polynomials to
the Hessian of the master function.

An interesting next step would be to establish similar results
for the elliptic KZB operators.

\medskip

The authors thank K. Styrkas for helpful discussions.

\section{Preliminaries}

\subsection{Notation}
Let $\g$ be a simple complex Lie algebra of rank $r$ with Cartan
subalgebra $\h$. Let $\Dl \subset \h^*$ be the set of roots, and
for $\al \in \Dl$, let $\g_\al \subset \g$ denote the root space
corresponding to $\al$.

Fix simple roots $\al_1, \dots, \al_r \in \Dl$.
Let $Q = \oplus_j \Z \al_j$ be the root lattice, and $Q_+$ the
elements of $Q$ with non-negative coefficients.
Let $\Dl_+ = \Dl \cap Q_+$ be the set of positive roots, and
$\Dl_- = \Dl \setminus \Dl_+$, its complement, the negative roots.
Let $\n_\pm = \oplus_{\al \in \Dl_\pm} \g_\al$ denote the positive
and negative root spaces.

Fix a nondegenerate $\g$-invariant bilinear form $(\, ,\,)$ on $\g$.
The form identifies $\g$ and $\g^*$ and defines a bilinear form
on $\g^*$. For a root $\al \in \Dl$, we denote its coroot by
$\al^\vee = 2\al/(\al,\al)$, and set $h_\al \in \h$ so that
$(h_\al, h) = \al^\vee(h)$ for all $h \in \h$.

For each $\al \in \Dl$, choose generators $e_\al \in \g_\al$
so that $(e_\al, e_{-\al}) = 1$. For positive roots $\al \in \Dl_+$,
we set $f_\al = \frac{2}{(\al_j, \al_j)} e_{-\al}$.
For a simple root $\al_j$, set $e_j = e_{\al_j}$, $f_j = f_{\al_j}$
and $h_j = h_{\al_j}$. Then $h_1, \dots, h_r,
e_1, \dots, e_r, f_1, \dots, f_r$ are Chevalley generators of $\g$.
Set $\rho = \frac{1}{2} \sum_{\al \in \Dl_+} \al$, so $\rho(h_j) = 1$
for $j=1, \dots, r$.

Denote by $U(\g)$ the universal enveloping algebra of $\g$.
The Chevalley involution $\omega$ is an automorphism of
$U(\g)$ defined by $\omega(e_j) = -f_j,\
\omega(f_j) = -e_j,\ \omega(h_j) = -h_j.$\
The antipode  $a$ is an anti-automorphism of $U(\g)$ defined by
$a(g) = - g$ for $g \in \g$.

\subsection{Shapovalov form}

The Poincar\'{e}-Birkhoff-Witt theorem gives the decomposition
\[
U(\g) = U(\n_-) \otimes U(\h) \otimes U(\n_+).
\]
Denote by $\gamma$ the projection $U(\g) \arr U(\h)$ along
$\n_- U(\g) + U(\g) \n_+$.

The Shapovalov form is the $U(\h)$-valued bilinear form
on $U(\g)$ defined by
\[
S(g_1,g_2) = \gamma( a\omega(g_1) g_2), \qquad g_1,\ g_2 \in U(\g).
\]
Using the identification of $U(\h)$ with $\C[\h^*]$,
the polynomials on $\h^*$, any $\mu \in \h^*$ defines a $\C$-valued
symmetric form $S_\mu$ on $U(\g)$ defined by evaluation at $\mu$.

For a $\g$-module $V$ and $\nu \in \h^*$, let
$V[\nu] = \{ u\in V: hu= \nu(h) u \text{ for } h \in \h\}$
denote the weight $\nu$ subspace. We consider modules with
weight decomposition, so $V = \otimes_\nu V[\nu]$.
A singular vector of weight $\nu$ in $V$ is $u \in V[\nu]$ such that
$\n_+ u = 0$. The space of all such vectors is denoted $\sing V[\nu]$.
For $\mu \in \h^*$, denote by $M_\mu$ the Verma module with highest
weight $\mu$, generated by the vector $\bs{1}_\mu \in \sing V[\mu]$.

We define the Shapovalov form on a Verma module $M_\mu$ by
\[
S(g_1 \bs{1}_\mu, g_2 \bs{1}_\mu) = S_\mu(g_1, g_2), \qquad g_1,\ g_2 \in U(\g).
\]
Weight subspaces $M_\mu[\nu]$ and $M_\mu[\eta]$ are orthogonal for
$\nu \neq \eta$ with respect to this form.
Let $M_\mu^*$ denote the restricted dual module to $M_\mu$,
with $\g$ acting by
$(g \phi)(v) = \phi (a(g) v)$ for $g\in \g$,
$\phi \in M_\mu^*$ and $v \in M_\mu$.
The Shapovalov form induces a map $S_\mu: M_\mu \arr M_\mu^*$
with $S_\mu(\bs{1}_\mu) = \bs{1}_\mu^*$, where $\bs{1}_\mu^*$ is
the lowest weight vector dual to $\bs{1}_\mu$, and with $S_\mu(g \bs{1}_\mu)
= \omega(g) \bs{1}^*_\mu$ for $g \in \g$.

The kernel of the Shapovalov form is the maximal proper submodule of $M_\mu$.
Denote by $L_\mu$ the irreducible quotient. The Shapovalov form induces a bilinear
form on $L_\mu$.  We shall use the Shapovalov form
on a tensor product, defined as the product of the Shapovalov forms
of factors.

For $\al \in \Dl_+$ and $k \in \{1, 2, \dots \}$, let
\begin{equation}
\chi^\al_k(\mu) = ( \al, \mu + \rho ) - \frac{k}{2}( \al, \al ).
\end{equation}
On the weight subspace $U(\n_-)[-\nu] \subset U(\n_-)$
the Shapovalov form has determinant \cite{Sh}
\[
{\rm det} S_\mu [-\nu] = {\rm const} \prod_{\al \in \Dl_+} \prod_{k = \N} \chi^{\al}_k (\mu)^{P(\nu - k\al)},
\]
where $S_\mu[-\nu]$ denotes the restriction to $U(\n_-)[-\nu]$, $\N$ are the numbers $\{1, 2, \dots \}$,
$P(\la)$ is the Kostant partition function and ${\rm const}$ is a nonzero constant depending on choice
of basis.
This formula also holds for $S_\mu [-\nu]$ defined as the restriction to $M_\mu[\mu - \nu]$,
so the Shapovalov form on $M_\mu$ is nondegenerate if
\begin{equation}
\label{eqn:shapgeneric}
\chi^\al_k(\mu) \neq 0 \qquad \text{ for all } \al \in \Dl_+, \ k \in \{1,2,\dots, \}.
\end{equation}

Let $S^{-1}_\mu [-\nu]$ be the inverse matrix to $S_\mu[-\nu]$.

\begin{lem}
\label{lem:inverseshap}
Let $\{F_j\}$ be a basis of $U(\n_-)[-\nu]$ with the first $N$ elements of $\{F_j \bs{1}_\mu\}$
form a basis of ${\rm Ker}(S_\mu[-\nu])$. Then an entry $(S^{-1}_\la[-\nu])_{j\ell}$ of 
$S^{-1}_\la[-\nu]$ is regular at $\la = \mu$ if $j > N$ or $\ell >N$.
\end{lem}
\begin{proof}
For $\mu$ such that $\chi^\al_k(\mu)$ is nonzero for every $\al \in \Dl_+$, $k \in \N$, 
the Shapovalov form is invertible, and each $(S^{-1}_\la[-\nu])_{j\ell}$ is regular.
We next consider $\mu$ such that $\chi^\al_k(\mu) = 0$ for exactly one $\al \in \Dl_+$ and 
$k \in \N$.
There is a unique proper submodule $M_{\mu - k \al}$ of highest weight $\mu - k \al$, 
with the dimension of the weight subspace $M_{\mu - k \al}[\mu - \nu]$ equal 
to $P(\nu - k \al)$.
For $\la = \mu + \epsilon \eta$ approaching $\mu$ transversely to $\chi^\al_k(\la) = 0$,  
the first $N = P(\nu - k \al)$ rows of $S_\la[-\nu]$ are divisible by $\epsilon$. 
Any minor $C_{j \ell}$ of $S_\la[-\nu]$ with $j>N$ is divisible by $\epsilon^N$, and 
since the determinant of $S_\la[-\nu]$ is divisible by exactly the $N$th power of $\epsilon$, 
the entry $(S^{-1}_\la[-\nu])_{j \ell}$ with $j >N$ is regular at $\epsilon = 0$.
By the symmetry of $S$, the entries with $\ell > N$ are also regular.

For general $\mu$, we note that if $\{ F_j \}$ is such that the first $N$ elements 
$\{ F_j \bs{1}_\mu \}$ form a basis of ${\rm Ker}(S_\mu[-\nu])$, then for $\la$ in a 
neighborhood of $\mu$, 
${\rm Ker}(S_\la[-\nu])$ is contained in the space spanned by 
the first $N$ elements of 
$\{ F_j \bs{1}_\la \}$.
Then we have that for $j>N$ or $\ell >N$, the meromorphic function 
$(S^{-1}_\la[-\nu])_{j \ell}$ is regular in a neighborhood of $\mu$ except 
perhaps on a subspace of codimension $2$. Thus it is regular in the 
entire neighborhood.
\end{proof}

The proof of this lemma was communicated to the authors by K. Styrkas. It mirrors
the proof in \cite{ES} that entries of $S^{-1}_\mu[-\nu]$ can have only simple poles.

\subsection{Singular vectors}
For $\mu \in \h^*$, let $I_\mu \subset U(\n_+)$ denote the annihilating ideal of
the vector $\bs{1}_\mu^* \in M_\mu^*$. Thus $\omega(I_\mu) \bs{1}_\mu$ equals ${\rm Ker}(S_\mu)$
in $M_\mu$.
Let $V$ be a $\g$ module such that for every $u \in V$, $\n_+ u$ is finite dimensional.
Denote by $V[\nu]_\mu \subset V[\nu]$ the subspace annihilated by $I_\mu$.
Let $\{F_j: j \geq 0\}$ be a homogeneous basis  of $U(\n_-)$ with $F_0 = \bs{1}$, the
identity element.

\begin{prop}
\cite{ES}
\label{prop:singonlyif}
There exists a singular vector in $M_\mu \otimes V[\mu + \nu]$
of the form $\bs{1}_\mu \otimes u + \sum_{j >0} F_j \bs{1}_\mu \otimes u_j$
for some $u_j \in V$ only if $u$ belongs to $V[\nu]_\mu$.
\end{prop}
\begin{proof}
The vector $\bs{1}_\mu \otimes u + \sum_{j>0} F_j \bs{1}_\mu \otimes u_j$
induces a map $M_{\mu}^* \arr V$ defined by
\[
\phi \mapsto \phi(\bs{1}_\mu) u + \sum_{j>0} \phi(F_j \bs{1}_\mu) u_j
\]
for $\phi \in M_\mu^*$. Since the vector is singular, this map
commutes with the action of $\n_+$. Clearly, it maps the lowest weight
vector $\bs{1}_\mu^*$ to $u$. The existence of such a map implies that
$I_\mu u$ is zero.
\end{proof}

For $u \in V[\nu]$,
set
\[
\Xi(\mu)(\bs{1} \otimes u) = \sum_{j,k \geq 0} (S_\mu^{-1})_{j\ell} F_j \otimes \omega(F_k)u
\]
in $U(\n_-) \otimes V$.
For $\mu$ satisfying (\ref{eqn:shapgeneric}), this vector is well-defined since $S_\mu$ is
non-degenerate.

\begin{prop}
\label{prop:singif}
For $u \in V[\nu]_\mu$, the vector $\Xi(\la)(\bs{1} \otimes u) \in U(\n_-) \otimes V$ as a
function of $\lambda$ is regular at $\la = \mu$.
\end{prop}
\begin{proof}
We consider a homogenous basis $\{F_j\}$ of $U(\n_-)$ that contains a 
subset $\{\tilde{F}_j\}$ such that $\{\tilde{F}_j \bs{1}_\mu \}$ forms 
a basis of ${\rm Ker}(S_\mu)$.
Since $u$ is in $V[\nu]_\mu$, any $\omega(\tilde{F}_j)u$ equals zero.
By Lemma \ref{lem:inverseshap}, the nonzero terms of $\Xi(\la)(\bs{1} \otimes u)$
are regular at $\la = \mu$.
\end{proof}

For $u \in V[\nu]_\mu$, the vector $\Xi(\mu)(\bs{1}_\mu \otimes u) \in M_\mu \otimes V[\mu + \nu]$
is singular \cite{ES}.

\begin{prop}
\label{prop:singdifference}
For $u \in V[\nu]_\mu$,
let $\bs{1}_\mu \otimes u + \sum_{j>0} F_j \bs{1}_\mu \otimes u_j$ be any
singular vector. Then the difference
\[
\left(\bs{1}_\mu \otimes u + \sum_{j>0} F_j \bs{1}_\mu \otimes u_j\right) -
\Xi(\mu)(\bs{1}_\mu \otimes u)
\]
lies in ${\rm Ker}(S_\mu) \otimes V$.
\end{prop}
\begin{proof}
The difference
is a singular vector, and the associated map $M_\mu^* \arr V$ maps $\bs{1}_\mu^* \mapsto 0$.
Since this map commutes with $\n_+$, all of $U(\n_+) \bs{1}_\mu^*$ maps to zero.
The image of $S_\mu$ in $M_\mu^*$ is $U(\n_+) \bs{1}_\mu^*$, so any vector of $M_\mu \otimes V$
orthogonal to it must belong to ${\rm Ker}(S_\mu) \otimes V$.
\end{proof}

We define the linear map $\mathcal{Q}(\mu): V[\nu]_{\mu -\rho -\frac{1}{2}\nu} \arr V[\nu]$ as
\[
\mathcal{Q}(\mu)u = \sum_{j,k \geq 0} \left(S_{\mu - \rho - \frac{1}{2}\nu}^{-1}\right)_{jk} a(F_j) \omega(F_k)u .
\]

\section{Gaudin and KZB operators}
Let $z = (z_1,\dots,z_n)$ be a collection of distinct complex numbers.
Let $V_1, \dots, V_n$ be $\g$-modules, and let $V = V_1 \otimes \dots \otimes V_n$.
For $x \in {\rm End}(V_p)$, let $x^{(p)} = 1 \otimes \dots \otimes 1 \otimes x \otimes 1 \otimes \dots \otimes 1$
be the endomorphism of $V$ with $g$ acting on the $p$th factor.
For $\sum_j x_j \otimes y_j \in {\rm End}(V_p) \otimes {\rm End}(V_s)$,
let $(\sum_j x_j \otimes y_j)^{(p,s)}$ denote $\sum_j x^{(p)}_j y^{(s)}_j$.

Let $\Omega \in \g\otimes \g$ be the symmetric invariant tensor
dual to $(\, ,\,)$. It has a decomposition $\Omega = \Omega_0 + \sum_{\al \in
\Delta} \Omega_\al$ where $\Omega_0 \in \h \otimes \h$ and $\Omega_\al \in \g_\al
\otimes \g_{-\al}$. Explicitly, for $\{ h_\nu \}$ an orthonormal basis of $\h$,
$\Omega_0 = \sum_\nu h_\nu \otimes h_\nu$, and for generators $e_\al \in \g_\al$
with $(e_\al, e_{-\al})=1$, $\Omega_\al = e_\al \otimes e_{-\al}$.
The Casimir element is $C = C_0 + \sum_{\al \in \Delta} C_{\al} \in U(\g)$, where
$C_0 = \sum_\nu h_\nu h_\nu$ and $C_\al = e_\al e_{-\al}$.

\subsection{Rational Gaudin operators}
The rational Gaudin operators are linear operators on $V$, given by
\[
K_p(z) = \sum_{s\neq p} \frac{\Omega^{(p,s)}}{z_p - z_s}, \hspace{.4in} p=1,\dots,n.
\]
Each $K_p(z)$ commutes with the action of $\g$ on $V$. For
 all $p,s$,
we have $[K_p(z), K_s(z)] = 0$.

\subsection{Trigonometric Gaudin operators}
Let
\bea
\Omega_+  \ = \ \frac{1}{2} \Omega_0 + \sum_{\al \in \Dl_+} \Omega_\al \ , \qquad \Omega_- \ = \
\frac{1}{2} \Omega_0 + \sum_{\al \in \Dl_-} \Omega_\al.
\eea
The trigonometric r-matrix is defined by
\[
r(x) = \frac{\Omega_+ x  + \Omega_-}{x-1}.
\]
For
$\xi \in \h$, the trigonometric Gaudin operators are defined as
\[
\mathcal{K}_p(z,\xi) = \xi^{(p)} + \sum_{s\neq p} r^{(p,s)}(z_p/z_s), \hspace{.4in} p = 1,\dots,n.
\]
Each $\mc{K}_p(z,\xi)$ commutes with the action of $\h$ on V and
 $[\mc{K}_p(z,\xi), \mc{K}_s(z,\xi)] =0$ for all $p,s$, see
 \cite{Ch, EFK}.

\subsection{KZB operators}
Let $H_+ \subset \C$ be the upper half plane, and
$\tau \in H_+$. Let $z_1, \dots, z_n \in \C$
be distinct modulo the lattice $\Z + \tau \Z$.
Let $\la \in \h$, with coordinates $\la = \sum \la_\nu h_\nu$
where $(h_\nu)$ is an orthonormal basis of $\h$. For given $z, \tau$, the KZB operators
$H_0$, $H_p$ are operators
 acting on functions $u(\la)$ with values in
$V[0] = (V_1 \otimes \dots \otimes V_n)[0]$, see \cite{FW}.

The KZB operators are
\[
H_0(z, \tau) = (4\pi i)^{-1} \triangle + \sum_{p,s} \frac{1}{2} \Gamma_\tau^{(p,s)} (\la, z_p - z_s, \tau),
\]
\[
H_p(z,\tau) = -\sum_\nu h_\nu^{(p)}
\partial_{\la_\nu} + \sum_{s:s\neq p} \Gamma_z^{(p,s)}(\la, z_p - z_s, \tau),
\qquad  p=1,\dots,n.
\]
Here $\triangle = \sum_v \partial_{\la_\nu}^2$
is the Laplace operator and the operators
$\Gamma_\tau(\la, z, \tau)$, $\Gamma_z(\la,z, \tau)$ are defined as follows.
For the Jacobi theta function
\[
\theta_1(t,\tau) = 2 e^{\frac{\pi i}{4} \tau} \sum_{j=0}^\infty (-1)^j e^{\pi i j (j+1) \tau} \sin((2j+1)\pi t)
\]
and the functions
\[
\rho(t, \tau) = \frac{\theta_1'(t,\tau)}{\theta_1(t,\tau)}, \hspace{.4in} \sigma_w(t,\tau) =
\frac{\theta_1(w-t,\tau)\theta_1'(0,\tau)}{\theta_1(w,\tau)\theta_1(t,\tau)},
\]
\[
\eta(t,\tau) = \rho(t,\tau)^2 + \rho'(t,\tau), \hspace{.4 in} \varphi(w,t,\tau) = \partial_w \sigma_w(t,\tau).
\]
we set
\[
\Gamma_\tau(\la, z, \tau) = \frac{1}{4\pi i} \eta(z,\tau) \Omega_0 - \frac{1}{2\pi i} \sum_{\al \in \Dl}
\varphi(\al(\la), z, \tau)  \Omega_\al,
\]
\[
\Gamma_z(\la , z, \tau) = \rho(z,\tau) \Omega_0 + \sum_{\al \in \Dl} \sigma_{\al(\la)}(z, \tau) \Omega_\al,
\]
so we have
\[
H_0(z,\tau) = \frac{1}{4 \pi i} \triangle + \frac{1}{4\pi i} \sum_{p,s} \left[\frac{1}{2}\eta(z_p - z_s,\tau) \Omega_0^{(p,s)}
- \sum_{\al \in \Dl} \varphi(\al(\la), z_p - z_s, \tau) \Omega_\al^{(p,s)} \right],
\]
\[
H_p(z,\tau) = -\sum_\nu h_\nu^{(p)} \partial_{\la_\nu} + \sum_{s:s\neq p} \left[ \rho(z_p - z_s,\tau) \Omega_0^{(p,s)}
+ \sum_{\al \in \Dl} \sigma_{\al(\la)}(z_p - z_s, \tau) \Omega_\al^{(p,s)} \right].
\]
By \cite{FW}
the operators $H_0(z,\tau), H_1(z,\tau), \dots, H_n(z,\tau)$
commute.

\subsection{Trigonometric KZB operators}
The trigonometric KZB operators are the limits of the operators
$H_0(z, \tau), H_1(z,\tau), \dots, H_n(z,\tau)$ as
$\tau \arr i \infty$.
\begin{prop}
The trigonometric KZB operators are
\[
H_0 = \frac{1}{4 \pi i} \triangle - \frac{1}{4 \pi i} \sum_{\al \in \Dl_+}
\frac{\pi^2}{\sin^2(\pi \al(\la))} \left( e_\al e_{-\al} + e_{-\al} e_\al \right),
\]
\[
H_p(z) = -\sum_\nu h_\nu^{(p)} \partial_{\la_\nu} + \pi \sum_{s:s\neq p} \left[
\cot(\pi(z_p-z_s)) \Omega^{(p,s)}
- \sum_{\al \in \Dl_+} \cot(\al(\la)) (\Omega_\al^{(p,s)} - \Omega_{-\al}^{(p,s)}) \right]
\]
for $p=1,\dots,n$.
\end{prop}
Note that $H_0$ does not depend on $z$.
\begin{proof}
We have that $\theta_1(t,\tau) = 2e^{\frac{\pi}{4} i \tau}
\left[\sin(\pi t) + O(e^{2\pi i \tau}) \right]$, so as $\tau \arr i \infty$
\[
\rho(t,\tau) \arr \frac{\pi \cos(\pi t)}{\sin(\pi t)}, \hspace{.4in} \sigma_w(t,\tau) \arr
\frac{\pi \sin(\pi (w - t))}{\sin(\pi w) \sin(\pi t)},
\]
\[
\eta(t,\tau) \arr \frac{\pi^2 \cos^2(\pi t)}{\sin^2(\pi t)} - \frac{\pi^2}{\sin^2(\pi t)} = -\pi^2,
\]
\[
\varphi(w,t,\tau) \arr
\frac{\pi^2}{\sin(\pi t)} \frac{\sin(\pi(w-w+t)}{\sin^2(\pi w)} = \frac{\pi^2}{\sin^2(\pi w)}.
\]
Thus, we have
\[
H_0 = \frac{1}{4 \pi i} \triangle + \frac{1}{4 \pi i} \sum_{p,s} \left[ \frac{-\pi^2}{2} \Omega_0^{(p,s)}
- \sum_{\al \in \Dl} \frac{\pi^2}{\sin^2(\pi \al(\la))} \Omega_\al^{(p,s)} \right].
\]
Since $\sum_{p,s} \Omega_0^{(p,s)}$ acts as zero on $V[0]$, we have that
\begin{align*}
H_0 & = \frac{1}{4 \pi i} \triangle - \frac{1}{4 \pi i} \sum_{p,s} \sum_{\al \in \Dl}
\frac{\pi^2}{\sin^2(\pi \al(\la))} \Omega_\al^{(p,s)} \\
		& = \frac{1}{4 \pi i} \triangle - \frac{1}{4 \pi i} \sum_{\al \in \Dl_+} \sum_{p,s}
\frac{\pi^2}{\sin^2(\pi \al(\la))} \left( \Omega_\al^{(p,s)} + \Omega_{-\al}^{(p,s)} \right).
\end{align*}
This gives
\[
H_0 = \frac{1}{4 \pi i} \triangle - \frac{1}{4 \pi i} \sum_{\al \in \Dl_+}
\frac{\pi^2}{\sin^2(\pi \al(\la))} \left( e_\al e_{-\al} + e_{-\al} e_\al \right)
\]
by summing over $p$ and $s$.

For $p=1,\dots,n$,
\[
H_p(z) = -\sum_\nu h_\nu^{(p)} \partial_{\la_\nu} + \sum_{s:s\neq p} \left[
\frac{\pi \cos(\pi(z_p-z_s))}{\sin(\pi(z_p-z_s))} \Omega_0^{(p,s)}
+ \sum_{\al \in \Dl} \frac{\pi\sin(\pi(\al(\la) - z_p + z_s))}{\sin(\pi\al(\la)) \sin(\pi(z_p-z_s))} \Omega_\al^{(p,s)} \right].
\]
\[
= -\sum_\nu h_\nu^{(p)} \partial_{\la_\nu} + \pi \sum_{s:s\neq p} \left[
\cot(\pi(z_p-z_s)) \Omega^{(p,s)}
- \sum_{\al \in \Dl_+} \cot(\pi\al(\la)) (\Omega_\al^{(p,s)} - \Omega_{-\al}^{(p,s)}) \right]
\]
as desired.
\end{proof}

For $s= 1,\dots, n$, let $Z_s$ denote $e^{-2 \pi i z_s}$. For $\beta \in \h^*$, and $\la \in \h$,
let $X_\beta(\la)$ denote $e^{-2\pi i \beta(\la)}$. 
Then for $p= 1, \dots, n$, we may write
\begin{equation}
\label{eqn:H_p}
H_p(z) = -\sum_\nu h_\nu^{(p)} \partial_{\la_\nu} - \pi i \sum_{s \neq p} \left[
\frac{Z_p + Z_s}{Z_p - Z_s} \Omega^{(p,s)} + \sum_{\al \in \Dl_+} \frac{1+X_\al}{1- X_\al}
\left(\Omega_\al^{(p,s)} - \Omega_{-\al}^{(p,s)} \right) \right].
\end{equation}

\subsection{KZ and KZB equations}
The rational, trigonometric Gaudin operators and the KZB operators are the right hand sides
of the rational, trigonometric KZ equations and the KZB equations respectively. Let $\kappa$
be a nonzero complex number. Then the rational KZ equations for a $V$-valued function
$u(z)$,
$z\in\C^n$, are
\[
\kappa \partial_{z_p} u(z) = K_p(z) u(z), \hspace{.4in} p=1,\dots,n.
\]
The trigonometric KZ equations for a $V$-valued function $u(z)$ are
\[
\kappa z_p \partial_{z_p} u(z)= \mathcal{K}_p (z,\xi) u(z), \hspace{.4in} p=1,\dots,n.
\]
The KZB equations for a $V[0]$-valued function $u(\la, z, \tau)$ are
\[
\kappa \partial_{z_p} u(\la, z, \tau) = H_p(z,\tau) u(\la, z,\tau), \hspace{.4in} p=1,\dots,n,
\]
\[
\kappa \partial_\tau u(\la, z, \tau)= H_0(z, \tau) u(\la, z,\tau),
\]
see \cite{KZ, Ch, EFK, FW}.

\section{Relations among Gaudin and KZB operators}
\subsection{Rational Gaudin and trigonometric Gaudin operators}
Let $V_1, \dots, V_n$ be highest weight $\g$-modules, with highest weights
$\La_1, \dots, \La_n$ respectively, and $V = V_1 \otimes \dots \otimes V_n$.
Let $M_\mu$ be the Verma module with highest weight $\mu$ generated by $\bs{1}_\mu$.

We label the factors in $M_\mu \otimes V$ starting with zero, so that the $p$th
factor is $V_p$ for $p=1, \dots, n$. Then as a consequence of Proposition 2.1
in \cite{MaV}, we have the following fact.

\begin{prop}
\label{prop:rattrig}
For $\mu \in \h^*$, $u \in V[\nu]_\mu$,
and $p = 1, \dots, n$,
\[
\left( z_p K_p(0,z_1,\dots, z_n) + \frac{1}{2}(\La_p, \La_p + 2\rho)\right) \Xi(\mu)(\bs{1}_\mu \otimes u)
\phantom{aaaaaaaaaaaaaaaaaa}
\]
\[
\phantom{aaaaaaa}
= \Xi(\mu)\left(\bs{1}_\mu \otimes \mc{K}_p\left(z_1,\dots, z_n, \mu + \rho + \frac{\nu}{2} \right) u\right)
\]
holds in $\sing M_{\mu} \otimes V[\nu + \mu]$.
\end{prop}
\begin{proof}
\cite{MaV}
The left side belongs to $\sing M_\mu \otimes V[\nu +\mu]$ since
$K_p(0,z_1,\dots,z_n)$ commutes with $\g$.
We rewrite the left side as
\[
\left( \Omega^{(p,0)} + \sum_{s \neq 0,p} \frac{z_p \Omega^{(p,s)}}{z_p - z_s} + \frac{1}{2} C^{(p)} \right)
\Xi(\mu)(\bs{1}_\mu \otimes u)
\]
using the fact that the Casimir element applied to $V_p$ is multiplication by the
constant $(\La_p, \La_p + 2\rho)$.
We have that $\frac{z_p \Omega^{(p,s)}}{z_p - z_s} = r^{(p,s)}(z_p/z_s) + \Omega_-^{(p,s)}$. We define
$C_\pm = \frac{1}{2}C_0 + \sum_{\al \in \Dl_\pm} C_\al$, so $C = C_+ + C_-$ and set $\Omega_\pm^{(p,p)} =
C_\pm^{(p)}$, which is just a choice of ordering. Then the previous expression is equal to
\[
\left( \sum_{s \neq 0,p} r^{(p,s)}(z_p/z_s) + \sum_{s=0}^n \Omega_-^{(p,s)} + \Omega_+^{(p,0)} + \frac{1}{2}
\left(C_+^{(p)} - C_-^{(p)}\right) \right) \Xi(\mu)(\bs{1}_\mu \otimes u).
\]
We note that $\sum_{s=0}^n \Omega_-^{(p,s)} = \frac{1}{2} \sum_{s=0}^n \Omega_0^{(p,s)}
+ \sum_{\al \in \Dl_+} e_{-\al}^{(p)}\left( \sum_{s=0}^n e_\al^{(s)} \right)$. Since $\Xi(\mu)(\bs{1}_\mu \otimes u)$
is singular, $\left( \sum_{s=0}^n e_\al^{(s)} \right) \Xi(\mu)(\bs{1}_\mu \otimes u) = 0$. Furthermore,
$\frac{1}{2} \left(\sum_{s=0}^n \Omega_0^{(p,s)}\right) \Xi(\mu)(\bs{1}_\mu \otimes u)$ is just
$\frac{1}{2} (\nu + \mu)^{(p)} \Xi(\mu)(\bs{1}_\mu \otimes u)$. We also note that $C_+ - C_- =
\sum_{\al \in \Dl_+} [e_\al, e_{-\al}]$ equals $2\rho$. We are left with
\begin{equation}
\label{eq:neartrig}
\left( \sum_{s \neq 0,p} r^{(p,s)}(z_p/z_s) + \frac{1}{2} (\nu + \mu)^{(p)} + \rho^{(p)} + \Omega_+^{(p,0)} \right)
\Xi(\mu)(\bs{1}_\mu \otimes u).
\end{equation}

For $\mu$ satisfying (\ref{eqn:shapgeneric}), each singular vector has an expansion
$\bs{1}_\mu \otimes v + \sum_{j,k > 0} (S_\mu^{-1})_{jk} F_j \otimes \omega(F_k)(\bs{1}_\mu \otimes v)$ for some nonzero $v \in V[\nu]$.
We must determine $v$ and can ignore the higher order terms.
The first three terms of (\ref{eq:neartrig}) do not act on the first factor $M_\mu$ at all. For the last, recall
that $\Omega_+^{(p,0)} = \frac{1}{2} \Omega_0^{(p,0)} + \sum_{\al \in \Dl_+} e_\al^{(p)} e_{-\al}^{(0)}$,
hence, $\Omega_+^{(p,0)} \Xi(\mu)(\bs{1}_\mu \otimes u) = \frac{1}{2} \mu^{(p)} \bs{1}_\mu \otimes u + higher \ order \ terms$.
So (\ref{eq:neartrig}) equals
\[
\Xi(\mu)(\bs{1}_\mu \otimes \left( \sum_{s \neq 0,p} r^{(p,s)}(z_p/z_s) + \frac{1}{2}\nu^{(p)} + \mu^{(p)} + \rho^{(p)}\right) u)
\]
as desired.

For any $\mu$ and $u\in V[\nu]_\mu$, we note that
\[
\left( \sum_{s \neq 0,p} r^{(p,s)}(z_p/z_s) + \frac{1}{2} (\nu + \la)^{(p)} + \rho^{(p)} + \Omega_+^{(p,0)} \right)
\Xi(\la)(\bs{1} \otimes u)
\]
is a $U(\h \oplus \n_-) \otimes V$-valued rational function of $\h^*$ regular at $\la = \mu$.
It is equal, modulo the annihilator of $\bs{1}_\la$, to
\[
\Xi(\la)\left(\bs{1} \otimes \mc{K}_p\left(z_1,\dots, z_n, \la + \rho + \frac{\nu}{2} \right) u\right),
\]
so
\[
\Xi(\mu)\left(\bs{1}_\mu \otimes \mc{K}_p\left(z_1,\dots, z_n, \mu + \rho + \frac{\nu}{2} \right) u\right)
\]
is defined, and the proposition holds.
\end{proof}

\begin{cor}
\label{cor:trigpreserves}
For $\mu \in \h^*$,
and $p = 1, \dots, n$,
the operator $\mc{K}_p\left(z, \mu + \rho + \frac{\nu}{2} \right)$
preserves $V[\nu]_\mu$.
\end{cor}
\begin{proof}
Since $\mc{K}_p(z, \mu + \rho + \frac{\nu}{2})$ commutes with $\h$, we need only
check that
\[
\Xi(\mu) \left(\bs{1}_\mu \otimes \mc{K}_p\left(z, \mu + \rho + \frac{\nu}{2} \right) u\right)
\]
is well-defined for $u \in V[\nu]_{\mu}$. By Proposition \ref{prop:rattrig}, it equals
\[
\left( z_p K_p(0,z_1,\dots, z_n) + \frac{1}{2}(\La_p, \La_p + 2\rho)\right) \Xi(\mu)(\bs{1}_\mu \otimes u)
\]
which is well-defined.
\end{proof}

\begin{cor}
\label{cor:singularandxi}
Let $\bs{1}_\mu \otimes u + \sum_{j>0} F_j \bs{1}_\mu \otimes u_j \in \sing M_\mu \otimes V[\mu + \nu]$
be an eigenvector of $K_p(0, z_1, \dots, z_p)$ with eigenvalue $\eps_p$ for $p =1, \dots, n$.
Then the leading term $u \in V[\nu]$ is an eigenvector of $\mc{K}_p(z_1, \dots, z_n, \mu + \rho + \frac{\nu}{2})$
with eigenvalue $z_p \eps_p + \frac{1}{2}(\La_p, \La_p + 2 \rho)$.
\end{cor}
\begin{proof}
By Proposition \ref{prop:singonlyif}, $u$ belongs to $V[\nu]_\mu$. By Proposition
\ref{prop:singif}, the vector $\Xi(\mu)(\bs{1}_\mu \otimes u)$ is well-defined.
By Proposition {prop:singdifference}, the difference
\[
\tilde{u} = \left( \bs{1}_\mu \otimes u + \sum_{j>0} F_j \bs{1}_\mu \otimes u_j \right) - \Xi(\mu)(\bs{1}_\mu \otimes u)
\]
belongs to $\sing {\rm Ker}(S_\mu) \otimes V [\mu +\nu]$. Since $U(\g)$ preserves ${\rm Ker}(S_\mu)$,
we have
\[
\left(z_p K_p(0,z_1,\dots, z_n) + \frac{1}{2}(\La_p, \La_p + 2\rho)\right) \tilde{u} \in {\rm Ker}(S_\mu) \otimes V
\]
so the leading term is zero. By Proposition \ref{prop:rattrig}, this leading term is
\[
\left(z_p \eps_p + \frac{1}{2}(\La_p, \La_p + 2\rho)\right) -
\mc{K}_p\left(z_1, \dots, z_n, \mu + \rho + \frac{\nu}{2}\right).
\]
\end{proof}

\subsection{Eigenfunctions of KZB operators in the trigonometric limit}
\label{sect:eigenfunctions}
We look for eigenfunctions to $H_0$ in the space $\mc{A}(\xi) \otimes V[0]$ defined as
follows.

For $\beta \in \h^*$, let $X_\beta(\la) = e^{-2\pi i \beta(\la)}$.
For a simple root $\al_j$, let $X_j = X_{\al_j}$.
Denote by $\mc{A}$ the algebra of functions on $\h$ which can be expressed
as meromorphic functions of variables $X_1, \dots, X_r$ with
poles only on $\cup_{\al \in \Dl} \{X_\al = 1\}$.
For $\xi \in \h^*$, we define $\mc{A}(\xi)$ to be the vector space of
functions of the form $e^{2\pi i \xi(\la)} \phi$
with $\phi \in \mc{A}$.

The algebra $\mc{A}$ is a subalgebra of the algebra
 $\C[[X_1, \dots, X_r]]$
of formal power series. \  For
$\phi(X_1, \dots, X_r) \in \mc{A}$, we define
 its leading term to be
${\rm const}_{\{X_j\}}(\phi) = \phi(0, \dots, 0)$.
For $e^{2\pi i \xi(\la)} \phi \in \mc{A}(\xi)$,
we define its  leading term to be $\phi(0, \dots, 0)$.

Let $E(\xi) \subset \mc{A}(\xi) \otimes V[0]$ be the
subspace of eigenvectors to $H_0$ with eigenvalue $\pi i (\xi, \xi)$.

\begin{prop}
\cite{FV2}
Let $\xi \in \h$ satisfy
\begin{equation}
\label{eqn:functiongeneric}
(\xi - \beta, \xi - \beta) \neq (\xi, \xi) \text{ for all nonzero } \beta \in Q_+\ .
\end{equation}
Then for $u \in V[0]$, there exists a unique $\psi_u^\xi = e^{2\pi i\xi(\la)}
\phi_u^\xi \in \mc{A}(\xi) \otimes V[0]$ such that
\bea
H_0 \psi_u^\xi = \pi i (\xi,\xi) \psi_u^\xi\quad
\text{and}\quad {\rm const}_{\{X_j\}}(\phi_u^\xi) = u .
\eea
\end{prop}

The proof of existence and uniqueness of $\psi_u^\xi$ as a formal power series is similar
to Heckman and Opdam \cite{HO}. Etingof \cite{E} gave a representation theoretic
construction of eigenfunctions
to $H_0$ for $\xi$ satisfying (\ref{eqn:functiongeneric}). Let $M_{\xi - \rho}$
be the Verma module of highest weight $\xi - \rho$, where $\xi$ satisfies
(\ref{eqn:shapgeneric}). Then for each $u \in V[0]$, there is a unique
homomorphism $\Phi_u: M_{\xi - \rho} \arr M_{\xi -\rho} \otimes V$ such that
$\Phi_u(\bs{1}_{\xi- \rho}) = \Xi(\xi - \rho)(\bs{1}_{\xi -\rho} \otimes u)$. Then as
formal power series,
\[
\psi_u^\xi(\la) = \frac{{\rm tr}_{M_{\xi -\rho}} \Phi_u \exp(2\pi i \la)}
{{\rm tr}_{M_{\xi - \rho}} \exp(2\pi i \la)} .
\]
Felder and Varchenko \cite{FV2} gave the explicit calculation of this
function and showed that it lies in $A(\xi) \otimes V[0]$. We recall this construction.

We first define a linear map $A_X: \ U(\n_-) \arr \mc{A} \otimes U(\n_-)$.
We set $A_X(\bs{1}) = \bs{1}$ for $\bs{1}$ the identity element of $U(\n_-)$.

For an element of $U(\n_-)$ of the form $F_{\beta_1} \cdots F_{\beta_m}$
where each $F_{\beta_k}$ belongs to  $\g_{-\beta_k}$ with $\beta_k \in \Dl_+$,
we set $A_X(F_{\beta_1} \cdots F_{\beta_m}) =
\sum_{\sigma \in S_m} A_X^\sigma(F_{\beta_1} \cdots F_{\beta_m})$,
where $S_m$ is the symmetric group on $m$ symbols, and $A_X^\sigma$ is defined by
\[
A_X^\sigma(F_{\beta_1} \cdots F_{\beta_m}) =
\prod_{k=1}^m \frac{X_{\beta_{\sigma(k)}}^{a^\sigma_k +1}}{1 - X_{\beta_{\sigma(1)}} \cdots X_{\beta_{\sigma(k)}}}
F_{\beta_{\sigma(1)}} \cdots F_{\beta_{\sigma(m)}}.
\]
For given $\sigma \in S_m$, the number $a^\sigma_k$ is defined to be
$\sum_{j=k}^{m-1} d^\sigma_j$, where
$d^\sigma_j = 1$ if $\sigma(j) > \sigma(j+1)$ and $d^\sigma_j = 0$ otherwise.

\begin{lem}
The operator $A_X$ is well-defined. In other words the relation
\[
A_X(F_{\beta_1} \cdots F_{\beta_\ell}F_{\beta_{\ell+1}} \cdots F_{\beta_m}) -
A_X(F_{\beta_1} \cdots F_{\beta_{\ell+1}}F_{\beta_\ell} \cdots F_{\beta_m}) =
A_X(F_{\beta_1} \cdots [F_{\beta_\ell}F_{\beta_{\ell+1}}] \cdots F_{\beta_m})
\]
holds for any collection $F_{\beta_1}, \dots, F_{\beta_m}$ with
each $\beta_k \in \Dl_+$ and $F_{\beta_k} \in \g_{-\beta_k}$.
\end{lem}
\begin{proof}
For each $\sigma \in S_m$, there exists some $\sigma'$ such that
\beq
\label{eqn:Asigmas}
A_X^\sigma(F_{\beta_1} \cdots F_{\beta_\ell}F_{\beta_{\ell+1}} \cdots F_{\beta_m}) -
A_X^{\sigma'}(F_{\beta_1} \cdots F_{\beta_{\ell+1}}F_{\beta_\ell} \cdots F_{\beta_m})
= B^\sigma(X) F_{\beta_{\sigma(1)}} \cdots F_{\beta_{\sigma(m)}}
\eeq
with $B^\sigma(X) \in \mc{A}$. In fact, $\sigma' = \tau_{\ell, \ell+1} \circ \sigma$ where
$\tau_{\ell,\ell+1}$ is the transposition of $\ell$ and $\ell+1$.
It is clear that
\[
A_X(F_{\beta_1} \cdots F_{\beta_\ell}F_{\beta_{\ell+1}} \cdots F_{\beta_m}) -
A_X(F_{\beta_1} \cdots F_{\beta_{\ell+1}}F_{\beta_\ell} \cdots F_{\beta_m}) =
\sum_{\sigma \in S_m} B^\sigma(X) F_{\beta_{\sigma(1)}} \cdots F_{\beta_{\sigma(m)}}.
\]

To calculate $B^\sigma(X)$, we note that the denominators of
$A^\sigma_X(F_{\beta_1} \cdots F_{\beta_\ell}F_{\beta_{\ell+1}} \cdots F_{\beta_m})$ and
$A^{\sigma'}_X(F_{\beta_1} \cdots F_{\beta_{\ell+1}}F_{\beta_\ell} \cdots F_{\beta_m})$ are
the same, since these only depend on the order of the $F_\beta$ appearing in
$F_{\beta_{\sigma(1)}} \cdots F_{\beta_{\sigma(m)}}$ and
$F_{\beta_{\sigma'(1)}} \cdots F_{\beta_{\sigma'(m)}}$.
The difference between $A^\sigma$ and $A^{\sigma'}$ comes from the coefficients
$a^\sigma_k$ and $a^{\sigma'}_k$, depending on $d^\sigma_j$ and
$d^{\sigma'}_j$ for $j \leq k$.

We first consider $\sigma$ for which the factors $F_{\beta_\ell}$ and
$F_{\beta_{\ell+1}}$ are not adjacent in the expression
$F_{\beta_{\sigma(1)}} \cdots F_{\beta_{\sigma(m)}}$,
in other words, $\sigma$ for which $|\sigma^{-1}(\ell) - \sigma^{-1}(\ell+1)| > 1$.
For such $\sigma$, exchanging $F_{\beta_\ell}$ and $F_{\beta_{\ell+1}}$
produces no change in the relative sizes of the indices of adjacent factors in
$F_{\beta_{\sigma(1)}} \cdots F_{\beta_{\sigma(m)}}$, so
$d^\sigma_j = d^{\sigma'}_j$ for all $j$ and
$a^\sigma_k = a^{\sigma'}_k$ for all $k$.
This tells us that for these $\sigma$, $B^\sigma(X) = 0$.

For the remaining $\sigma$,
$F_{\beta_\ell}$ and $F_{\beta_{\ell+1}}$ are adjacent in
$F_{\beta_{\sigma(1)}} \cdots F_{\beta_{\sigma(m)}}$. We first consider $\sigma$ such
that
$F_{\beta_\ell} F_{\beta_{\ell+1}}$ appears, or $\sigma^{-1}(\ell) + 1 = \sigma^{-1}(\ell+1)$.
We calculate $d^\sigma_{\sigma^{-1}(\ell)} = 0$, since
$\sigma(\sigma^{-1}(\ell))$ is less than $\sigma(\sigma^{-1}(\ell) + 1)$.
We see that
$d^{\sigma'}_{\sigma^{-1}(\ell)} = d^{\sigma'}_{\sigma'^{-1}(\ell+1)} = 1$.
For all
$j \neq \sigma^{-1}(\ell)$, it is clear that $d^\sigma_j = d^{\sigma'}_j$.
Thus $a^{\sigma'}_k = a^\sigma_k + 1$ for $k \leq \sigma^{-1}(\ell)$ and
$a^{\sigma'}_k = a^\sigma_k$ for $k > \sigma^{-1}(\ell)$. For such $\sigma$,
\[
A^{\sigma'}_X(F_{\beta_1} \cdots F_{\beta_{\ell+1}}F_{\beta_\ell} \cdots F_{\beta_m})
= X_{\beta_{\sigma(1)}} \cdots X_{\beta_{\sigma(\sigma^{-1}(\ell))}}
A^\sigma_X(F_{\beta_1} \cdots F_{\beta_{\ell}}F_{\beta_{\ell+1}} \cdots F_{\beta_m})
\]
so
\[
B^\sigma(X) F_{\beta_{\sigma(1)}} \cdots F_{\beta_{\sigma(m)}} =
(1 - X_{\beta_{\sigma(1)}} \cdots X_{\beta_{\sigma(\sigma^{-1}(\ell))}})
A^\sigma_X(F_{\beta_1} \cdots F_{\beta_{\ell}}F_{\beta_{\ell+1}} \cdots F_{\beta_m}).
\]

We must also consider $\sigma$ such that $F_{\beta_{\ell+1}} F_{\beta_{\ell}}$ appears
in $F_{\beta_{\sigma(1)}} \cdots F_{\beta_{\sigma(m)}}$. For these $\sigma$,
we see that $\sigma'$ is in the previous category.
Then $B^\sigma(X) = - B^{\sigma'}(X)$, so
\[
B^\sigma(X) F_{\beta_{\sigma(1)}} \cdots F_{\beta_{\sigma(m)}}
+ B^{\sigma'}(X) F_{\beta_{\sigma'(1)}} \cdots F_{\beta_{\sigma'(m)}}
B^\sigma(X) F_{\beta_{\sigma(1)}} \cdots [F_{\beta_\ell}, F_{\beta_{\ell+1}}]
\cdots F_{\beta_{\sigma(m)}}.
\]

We can write $B^\sigma(X)$ more explicitly
with variables $Y_{\beta_k}$, where $Y_{\beta_k} = X_{\beta_k}$ for
$k < \ell$, $Y_{\beta_\ell} = X_{\beta_\ell} X_{\beta_{\ell + 1}}$, and
$Y_{\beta_{k}} = X_{\beta_{k+1}}$ for $k > \ell$. Then
\[
B^\sigma(X) = \prod_{k=1}^{m-1} \frac{Y_{\beta_{\widehat{\sigma}(k)}}^{b^{\widehat{\sigma}}_k}}
{1 - Y_{\beta_{\widehat{\sigma}(1)}} \cdots Y_{\beta_{\widehat{\sigma}(k)}}}
\]
where $\widehat{\sigma}$ is the element of $S_{m-1}$ that acts like $\sigma$, but
treats $\ell$ and $\ell+1$ as a unit, and $b^{\widehat{\sigma}}_k$ is a sum of the
first $k$ numbers $d^\sigma_j$, skipping $d^\sigma_{\sigma^{-1}(\ell)}$, which is zero.
In fact,
\[
B^\sigma(X) F_{\beta_{\sigma(1)}} \cdots [F_{\beta_\ell}, F_{\beta_{\ell+1}}]
\cdots F_{\beta_{\sigma(m)}}
=
A^{\widehat{\sigma}}_Y(F_{\beta_1} \cdots
[F_{\beta_\ell}, F_{\beta_{\ell+1}}] \cdots F_{\beta_m}),
\]
and
\[
\sum_{\widehat{\sigma}\in S_m} A^{\widehat{\sigma}}_Y(F_{\beta_1} \cdots
[F_{\beta_\ell}, F_{\beta_{\ell+1}}] \cdots F_{\beta_m}) = A_Y(F_{\beta_1} \cdots
[F_{\beta_\ell}, F_{\beta_{\ell+1}}] \cdots F_{\beta_m}).
\]
\end{proof}

\begin{prop}
\cite{FV2}
\label{prop:eigenfunct}
For $\bs{1}_{\xi - \rho} \otimes u + \sum_{j>0} F_j \otimes u_j \in
\sing M_{\xi - \rho} \otimes V [0]$,
the function
\bea
\psi(\la) = e^{2 \pi i \xi(\la)}(u + \sum_{j>0} A_X(F_{j}) u_{j})
\eea
belongs to $E(\xi)$.
\end{prop}

For $u \in V[0]_{\xi - \rho}$, we denote the function associated to
$\Xi(\xi - \rho)(\bs{1}_{\xi - \rho} \otimes u)$ by
\[
\psi^\xi_u = e^{2 \pi i \xi(\la)}(u + \sum_{j>0} (S^{-1}_{\xi-\rho})_{jk} A_X(F_{j}) \omega(F_k) u_{j}).
\]
Let $P^\xi: V[0]_{\xi -\rho} \arr E(\xi)$ denote the map  $u \mapsto \psi^\xi_u$.

\begin{prop}
\cite{FV2}
For $\xi \in \h^*$ satisfying (\ref{eqn:functiongeneric}) the map $P^\xi$ is an isomorphism
between $V[0]$ and $E(\xi)$.
\end{prop}

\begin{prop}
\label{prop:KZBpreserves}
For $p=1,\dots,n$, the operator
 $H_p(z)$ preserves $E(\xi)$.
\end{prop}
\begin{proof}
We use the formulation of $H_p(z)$
 in formula (\ref{eqn:H_p}).
Differentiations with respect to $\la$ preserve $\mc{A}(\xi)$, so  $H_p(z)$ preserves
$\mc{A}(\xi) \otimes V[0]$. Since $H_0$ and $H_p(z)$ commute,
$H_p(z)$ also preserves $E(\xi)$.
\end{proof}

\subsection{Trigonometric Gaudin and trigonometric KZB operators}
\begin{lem}
\label{lem:trigeigenfunct}
For $\xi$ with property
(\ref{eqn:functiongeneric}) and
$p=1, \dots, n$, we have
\[
H_p(z_1, \dots, z_n) P^\xi =
P^\xi \mc{K}_p(e^{-2\pi i z_1}, \dots, e^{-2\pi i z_n}, \xi)
\]
as maps from $V[0]$ to $E(\xi)$.
\end{lem}
\begin{proof}
For any $u \in V[0]$, we have $H_p(z) P^\xi(u) = H_p(z) \psi^\xi_u \in E(\xi)$ by Proposition
\ref{prop:KZBpreserves}. Since $\xi$ satisfies (\ref{eqn:functiongeneric}),
$H_p(z) \psi^\xi_u$ equals $\psi^\xi_v(\la)$ for some $v \in V[0]$.
By formula (\ref{eqn:H_p})
we have
\bea
v & = & -2\pi i \xi^{(p)}u - \pi i  \sum_{s\neq p}
\left[\frac{Z_p+Z_s}{Z_p - Z_s} \Omega^{(p,s)} + \sum_{\al \in \Dl_+}
\left(\Omega_\al^{(p,s)} - \Omega_{-\al}^{(p,s)} \right) \right]u \\
& = & -2\pi i \xi^{(p)}u -\pi i \sum_{s\neq p}\left[\frac{Z_p + Z_s}{Z_p - Z_s}(\Omega_+^{(p,s)} + \Omega_-^{(p,s)})
+ (\Omega_+^{(p,s)} - \Omega_-^{(p,s)}) \right]u \\
& = & -2\pi i \xi^{(p)}u -2\pi i \sum_{s\neq p} r^{(p,s)} \left( \frac{Z_p}{Z_s} \right)u \\
& = & -2\pi i \mathcal{K}_p(Z_1,\dots,Z_n,\xi) u .
\eea
\end{proof}

The lemma holds for generic $\xi$ and so implies the following result.

\begin{cor}
For $\xi \in \h^*$ and $u \in V[0]_{\xi - \rho}$, we have
\[
H_p(z_1, \dots, z_n) \psi^\xi_u = \psi^\xi_v
\]
for $v = \mc{K}_p(e^{-2\pi i z_1}, \dots, e^{-2\pi i z_n}, \xi)u$.
\end{cor}

\section{Scalar products}
\subsection{Shapovalov form and rational Gaudin operators}
We recall the following fact, see for example \cite{RV}.
\begin{prop}
\label{prop:ratsym}
\
Let $V = V_1 \otimes \dots \otimes V_n$, and let $S$ be the tensor Shapovalov form
on V. Let $u,v \in V$.
Then
\[
S(K_p(z_1,\dots,z_n) u, v) = S(u, K_p(z_1,\dots,z_n) v),
\qquad
 p=1,\dots n.
\]
\end{prop}

\subsection{Scalar product and trigonometric Gaudin operators}
We define the following family of bilinear forms
on $V$, depending on a parameter $\xi \in \h^*$.

\begin{defn}
\label{def:trigprod}
For $\xi, \nu \in \h^*$ introduce a bilinear form $\langle , \rangle_\xi$ on
$ V[\nu]_{\xi - \rho - \frac{1}{2}\nu}$ by the formula
\begin{equation}
\label{eqn:trigprod}
\left\langle u,v \right\rangle_\xi = S(u, \mathcal{Q}(\xi) v).
\end{equation}
\end{defn}

\begin{prop}
\label{prop:rattrigproducts}
The bilinear form $\langle \, , \, \rangle_\xi$
on $V[\nu]_{\xi -\rho - \frac{1}{2}\nu}$ and the Shapovalov form on \\
$M_{\xi - \rho - \frac{1}{2}\nu} \otimes V$ satisfy the relation
\[
\langle u, v \rangle_\xi = S\left(\Xi(\xi - \rho - \frac{1}{2}\nu)
(\bs{1}_{\xi - \rho - \frac{1}{2}\nu} \otimes u), \Xi(\xi - \rho - \frac{1}{2}\nu)
(\bs{1}_{\xi - \rho - \frac{1}{2}\nu} \otimes v)\right).
\]
\end{prop}
\begin{proof}
Set $\mu = \xi - \rho - \frac{1}{2}\nu$.
Then
\[
\Xi(\mu)(\bs{1}_\mu \otimes u) = \bs{1}_\mu \otimes u
+ \sum_{j,k>0} (S_\mu^{-1})_{jk} \left(F_j \otimes \omega(F_k)\right)
(\bs{1}_\mu \otimes u)
\]
with respect to a homogeneous basis ${F_j}$ of $U(\n_-)$. Thus the right side of the
equation (\ref{eqn:trigprod}) is
\[
\sum_{j,k,l,m \geq 0} (S_\mu^{-1})_{jk} (S_\mu^{-1})_{lm}
S\left(F_j \bs{1}_\mu \otimes \omega(F_k)u, F_l \bs{1}_\mu \otimes \omega(F_m) v \right),
\]
or
\[
\sum_{j,k,l,m \geq 0} (S_\mu^{-1})_{jk} (S_\mu^{-1})_{lm} (S_\mu)_{jl} S(\omega(F_k)u,\omega(F_m) v)
= \sum_{k,m \geq 0} (S_\mu^{-1})_{km} S( u, a(F_k) \omega(F_m) v).
\]
This last expression is just $S(u, \mathcal{Q}(\xi)w)$.
\end{proof}

\begin{cor}
\label{cor:singtrigprod}
Let $\bs{1}_\mu \otimes u + \sum_{j>0} F_j \otimes u_j$ and
$\bs{1}_\mu \otimes v + \sum_{j>0} F_j \otimes v_j$ be two vectors in
$\sing M_\mu \otimes V[\mu + \nu]$. Then the relation
\[
S\left(\bs{1}_\mu \otimes u + \sum_{j>0} F_j \otimes u_j,
\bs{1}_\mu \otimes v + \sum_{j>0} F_j \otimes v_j \right)
= \langle u, v \rangle_{\mu + \rho + \frac{1}{2}\nu}
\]
holds.
\end{cor}
\begin{proof}
The singular vectors differ from $\Xi(\mu)(\bs{1}_\mu \otimes u)$ and
$\Xi(\mu)(\bs{1}_\mu \otimes v)$ by vectors in ${\rm Ker}(S_\mu)\otimes V$.
\end{proof}

\begin{cor}
The bilinear form $\langle \, , \, \rangle_\xi$
on $V[\nu]_{\xi -\rho - \frac{1}{2}\nu}$ is symmetric.
\end{cor}

\begin{thm}
\label{thm:trigsym}
For $\xi \in \h^*$, $u,v \in V[\nu]_{\xi - \rho - \frac{1}{2}\nu}$
and $p= 1,\dots,n$, we have
\[
\langle \mathcal{K}_p(z,\xi) u, v \rangle_\xi = \langle u, \mathcal{K}_p(z,\xi)v \rangle_\xi.
\]
\end{thm}
\begin{proof}
Note that both sides are defined by Corollary \ref{cor:trigpreserves}.
By Proposition \ref{prop:ratsym}, we know that $z_p K_p(0,z_1,\dots,z_n) +
\frac{1}{2}(\La_p, \La_p + 2\rho)$ is symmetric with respect to the Shapovalov
form defined on $M_{\xi - \rho - \frac{1}{2}\nu} \otimes V$. This and
Proposition \ref{prop:rattrig} imply that
\[
S(\Xi(\mu)(\bs{1}_\mu \otimes \mc{K}_p(z, \xi) u), \Xi(\mu)(\bs{1}_\mu \otimes v))
= S(\Xi(\mu)(\bs{1}_\mu \otimes u), \Xi(\mu)(\bs{1}_\mu \otimes \mc{K}_p(z,\xi) v))
\]
for $\mu = \xi - \rho - \frac{1}{2}\nu$. Then by Proposition \ref{prop:rattrigproducts}
we have the theorem.
\end{proof}

\subsection{Scalar product and KZB operators}

\begin{defn}
For $\xi \in \h^*$, we define the bilinear form $\langle \, , \, \rangle$
on $\oplus_{\beta \in Q} \mc{A}(\xi + \beta) \otimes V[0]$ by
\begin{equation}
\label{eqn:integraldef}
\langle \psi_1(\la), \psi_2(\la) \rangle =
\int_{C} S(\psi_1(\la), \psi_2(-\la)) d\la_1 d\la_2 \cdots d\la_r
\end{equation}
where $C$ is given by each $\la_j = (\la, \al_j)$ ranging along the interval
from $0 - i \delta$
to $1 - i \delta$ for some $\delta > 0$.
\end{defn}

\begin{prop}
The bilinear form $\langle \, , \, \rangle$ is well-defined on
$\oplus_{\beta \in Q} \mc{A}(\xi + \beta) \otimes V[0]$.
\end{prop}
\begin{proof}
Recall that $\mc{A}(\mu)$ is defined as
the space of functions on $\h$
of the form $\psi(\la) = e^{2 \pi i \mu(\la)} \phi(X)$ where
$X = (X_1, \dots, X_r)$, $X_j = e^{-2\pi i \al_j(\la)}$, and
$\phi$ is a meromorphic function with poles only
on the hyperplanes $\cup_{\al \in \Dl} \{X_\al = 1\}$.
For $\psi_1(\la) \in \mc{A}(\xi) \otimes V[0]$
and $\psi_2(\la) \in \mc{A}(\xi +\beta) \otimes V[0]$
with $\beta = \sum_{j=1}^r b_j \al_j \in Q$
the bilinear form is
\[
\langle \psi_1(\la), \psi_2(\la) \rangle =
\int_{C} S(e^{2\pi i \xi(\la)} \phi_1(X), e^{-2\pi i \xi(\la)}e^{-2\pi i \beta(\la)}
\phi_2(X^{-1}))  d\la_1 d\la_2 \cdots d\la_r
\]
where $X^{-1}$ denotes $(X_1^{-1}, \dots, X_r^{-1})$.
The factor $e^{-2\pi i \beta(\la)}$ is $X_1^{b_1} \cdots X_r^{b_r}$, so the integrand is periodic
and we may write
\begin{equation}
\label{eqn:scalprodX}
\langle \psi_1(\la), \psi_2(\la) \rangle = (-2 \pi i)^{-r} \int_{\tilde{C}} X_1^{b_1} \cdots X_r^{b_r}
S(\phi_1(X), \phi_2(X^{-1})) \frac{dX_1}{X_1}\wedge\dots\wedge \frac{ dX_r}{X_r}
\end{equation}
where $\tilde{C}$ is a torus
$\{ X\ |\ | X_j | = \epsilon, j =1, \dots, r \}$ with
$\epsilon < 1$.
The torus
$\tilde{C}$ doesn't cross the poles of $\phi_1(X)$ or $\phi_2(X^{-1})$.
\end{proof}

\begin{prop}
\label{prop:H0symm}
The trigonometric KZB operators $H_0, H_1, \dots, H_n$ are symmetric with respect to
$\left\langle \  , \right\rangle$ on $\oplus_{\beta \in Q} \mc{A}(\xi + \beta) \otimes V[0]$.
\end{prop}
\begin{proof}
Recall that $H_0$ is
\[
H_0 = \frac{1}{4 \pi i} \triangle - \frac{1}{4\pi i} \sum_{\al \in \Dl_+} \frac{\pi^2}{\sin^2(\pi \al(\la))}
(e_\al e_{-\al} + e_{-\al} e_\al).
\]
Since the integrand of (\ref{eqn:integraldef}) is periodic, the Laplace operator $\triangle$ is symmetric
with respect to $\left\langle \  , \right\rangle$.
For each $\al$, the operator $e_\al e_{-\al}$ is adjoint to $e_{-\al} e_\al$ with respect to Shapovalov form, and
$\sin^{-2}(\pi\al(\la))$ is an even function of $\la$, so these terms are symmetric.

For $p = 1, \dots, n$, the operator $H_p(z)$ is given by the formula
\[
H_p(z) = -\sum_\nu h_\nu^{(p)} \partial_{\la_\nu} + \pi \sum_{s:s\neq p} \left[
\cot(\pi(z_p-z_s)) \Omega^{(p,s)}
- \sum_{\al \in \Dl_+} \cot(\al(\la)) (\Omega_\al^{(p,s)} - \Omega_{-\al}^{(p,s)}) \right].
\]
Each $\partial_{\la_\nu}$ is symmetric by integration by parts, and $h_\nu^{(p)}$ is symmetric
with respect to Shapovalov form. The operator $\Omega^{(p,s)}$ is the symmetric invariant tensor
acting on the $p$th and $s$th factor, and is symmetric with respect to Shapovalov form. Each
$\Omega_\al^{(p,s)}$ is adjoint to $\Omega_{-\al}^{(p,s)}$ with respect to Shapovalov form. Since
$\cot(\al(\la))$ is an odd function of $\la$, each $\cot(\al(\la)) (\Omega_\al^{(p,s)} - \Omega_{-\al}^{(p,s)})$
is self-adjoint with respect to $\langle \, , \, \rangle$.
\end{proof}

In fact, the elliptic KZB operators are symmetric with respect to $\langle \, , \, \rangle$ as well.

\subsubsection{Eigenfunctions of $H_0$}

\begin{prop}
\label{prop:crosstermsorth}
Let $\beta \in Q$ be nonzero. For $\xi \in \h^*$,
$u \in V[0]_{\xi-\rho}$ and $v \in V[0]_{\xi + \beta - \rho}$,
we have
\[
\langle \psi^\xi_u, \psi^{\xi + \beta}_v \rangle =0.
\]
\end{prop}
\begin{proof}
For $\xi$ satisfying (\ref{eqn:functiongeneric}),
$\psi^\xi_u$ and $\psi^{\xi+\beta}_v$ have different eigenvalues
with respect to $H_0$, so Proposition \ref{prop:H0symm} implies
they are orthogonal. The product
$\langle \psi^\xi_u, \psi^{\xi + \beta}_v \rangle$
is analytic in $\xi$, since the functions $\psi^\xi_u$ and $\psi^{\xi+\beta}_v$
are analytic, so the proposition holds for all $\xi$.
\end{proof}

For generic $\xi$, Proposition \ref{prop:crosstermsorth} implies that the
spaces $E(\xi)$ and $E(\xi + \beta)$ are orthogonal for $\beta \in Q$ nonzero.
For $\xi$ such that $(\xi, \al_j^\vee)$ is an integer for some simple root $\al_j$
and certain $\beta$, the spaces $E(\xi)$ and $E(\xi + \beta)$ are not disjoint,
so are not orthogonal.

Recall that the definition of $A_X$ is $A_X(\bs{1}) = \bs{1}$ and
\[
A_X(F_{\beta_1} \cdots F_{\beta_m}) =
\sum_{\sigma \in S_m} A_X^\sigma(F_{\beta_1} \cdots F_{\beta_m})
\]
with
\[
A_X^\sigma(F_{\beta_1} \cdots F_{\beta_m}) = \prod_{k=1}^m
\frac{X^{a_k^\sigma + 1}_{\beta_{\sigma(k)}}}
{1 - X_{\beta_{\sigma(1)}} \cdots X_{\beta_{\sigma(k)}}} F_{\beta_1} \cdots F_{\beta_m}
\]
where $a^\sigma_k$ is defined as the cardinality of the subset of $\{k, \dots, m-1\}$
consisting of those $j$ satisfying $\sigma(j) > \sigma(j+1)$.
It is clear that $A_X(F_{\beta_1} \cdots F_{\beta_m})$ is zero at $X=0$ if $m>0$.

\begin{lem}
\label{lem:Aantipode}
Let $X^{-1}$ denote $(X_1^{-1}, \dots, X_r^{-1})$. The map
$A_{X^{-1}}: U(\n_-) \arr \mc{A} \otimes U(\n_-)$ is regular at $X=0$ with
\[
\lim_{X=0} A_{X^{-1}} = a,
\]
where $a$ is the antipode.
\end{lem}
\begin{proof}
We have the formula
\[
A_{X^{-1}}^\sigma(F_{\beta_1} \cdots F_{\beta_m}) =
\prod_{k=1}^m \frac{X_{\beta_{\sigma(k)}}^{m-k - a^\sigma_k}}
{X_{\beta_{\sigma(1)}} \cdots X_{\beta_{\sigma(k)}} - 1}
F_{\beta_{\sigma(1)}} \cdots F_{\beta_{\sigma(m)}}.
\]
Since $a_k^\sigma$ equals the cardinality of a subset of
$\{k, \dots, m-1\}$, the expression is regular at $X=0$.
In fact, $A_{X^{-1}}^\sigma(F_{\beta_1} \cdots F_{\beta_m})$
is nonzero only if $a^\sigma_k = m-k$ for every $k$. This holds
only if $\sigma$ is the permutation sending each $k$ to $m-k+1$.
Denoting this permutation by $\sigma_0$, we have
\[
\lim_{X=0} A_{X^{-1}}^{\sigma_0}(F_{\beta_1} \cdots F_{\beta_m}) =
(-1)^m F_{\beta_m} \cdots F_{\beta_2} F_{\beta_1},
\]
which is the antipode map $a$ on $U(\n_-)$.
\end{proof}

\begin{prop}
\label{prop:eigentrigprod}
\cite{EV2}
For $\xi \in \h^*$ and $u, v \in V[0]_{\xi-\rho}$,
we have the relation
\[
\langle \psi^\xi_u, \psi^\xi_v \rangle =
\langle u, v \rangle_\xi.
\]
\end{prop}
\begin{proof}
Recall that by Proposition \ref{prop:eigenfunct}, for $u \in V[0]_{\xi-\rho}$, $\psi^\xi_u$
has the form
$\psi^\xi_u(\la)=$
\linebreak
$e^{2\pi i \xi(\la)}\left(\sum_j A_X (F_j) u_j \right)$ with the vectors
$u_j$ defined by the condition
$\Xi(\xi - \rho)(\bs{1}_{\xi-\rho} \otimes u) = \sum_{j\geq 0} F_j \bs{1}_{\xi - \rho} \otimes u_j$
for a homogeneous basis $\{F_j\}$ of $U(\n_-)$.
It follows that $\psi^{\xi}_v(-\la)$
has the form
\[
\psi^{\xi}_v(-\la) = e^{-2 \pi i \xi(\la)}\left(\sum_{k \geq 0} A_{X^{-1}} (F_k) v_k \right)
\]
where $X^{-1} = (X_1^{-1}, \dots, X_r^{-1})$,
and $v_k$ is defined by the condition
\[
\Xi(\xi - \rho)(\bs{1}_{\xi -\rho} \otimes v) = \sum_{k \geq 0} F_k \bs{1}_{\xi - \rho} \otimes v_k .
\]

Formula (\ref{eqn:scalprodX}) gives in this case
\[
\langle \psi^\xi_u(\la), \psi^\xi_v(\la) \rangle = (-2 \pi i)^{-r} \int_{\tilde{C}}
S\left( \sum_j A_X (F_j) u_j  , \sum_k A_{X^{-1}} (F_k) v_k \right) \frac{dX_1}{X_1}\wedge\dots\wedge \frac{ dX_r}{X_r}.
\]
The expression $\sum_{j\geq 0} A_X (F_j) u_j$ is regular at $X=0$ with value $u$. Lemma \ref{lem:Aantipode}
gives that $\sum_{k \geq 0} A_{X^{-1}} (F_k) v_k$ is regular at $X=0$ with value
$\sum_{k \geq 0} a(F_k) v_k$, where $a$ denotes the antipode map. The integration is just evaluation
at $X=0$:
\[
\langle \psi^\xi_u(\la), \psi^\xi_v(\la) \rangle = S(u, \sum_{k \geq 0} a(F_k) v_k).
\]

The definition of $\Xi(\xi - \rho)(\bs{1}_{\xi-\rho} \otimes v)$
gives each $v_k$ as
\[
v_k = \sum_{\ell \geq 0} (S_{\xi-\rho}^{-1})_{k \ell} \omega(F_\ell)v,
\]
so we have
\[
\langle \psi^\xi_u(\la), \psi^\xi_v(\la) \rangle = S\left( u, \sum_{k, \ell \geq 0} (S_{\xi-\rho}^{-1})_{k \ell}
a(F_k) \omega(F_\ell) v \right).
\]
The expression $\sum_{k, \ell \geq 0} (S_{\xi-\rho}^{-1})_{k \ell}
a(F_k) \omega(F_\ell) v$ is the definition of $\mc{Q}(\xi)v$, so
\[
\langle \psi^\xi_u(\la), \psi^\xi_v(\la) \rangle = S(u, \mc{Q}(\xi)v)
\]
holds, which gives the proposition.
\end{proof}

\section{Bethe ansatz}
\label{sec Bethe}
\subsection{Rational Gaudin Model}
Let $V = V_1 \otimes \dots \otimes V_n$, where $V_p$ are irreducible highest weight
$\g$-modules of highest weight $\La_p$ with highest weight vectors $v_p$.
Set $\mathbf{\La}  = (\La_1,\dots, \La_n)$ and $\La = \sum_p \La_p$.
Let $\mathbf{m} = (m_1,\dots, m_r)$ be
a collection of non-negative integers, $m = \sum_j m_j$ and
$\mathbf{m}_\al = \sum_{j=1}^r m_j \al_j$.

The Bethe ansatz gives
simultaneous eigenvectors to the operators $K_p(z_1,\dots,z_n)$
on $\sing V[\La - \mathbf{m}_\al]$.

\subsubsection{Master function}
Let
\[
t = (t_1^{(1)},\dots, t_{m_1}^{(1)}, t_1^{(2)},\dots, t_{m_2}^{(2)}, \dots, t_1^{(r)},
\dots, t_{m_r}^{(r)}) \in \mathbb{C}^m.
\]
We express this ordering of coordinates
 as $(j,k) < (j',k')$ if
either
 $j < j'$ or
  $j = j'$ and $k <k'$;
  here
$(j,k)$ corresponds to $t^{(j)}_k$.

The master function $\Phi_K(t,z,\mathbf{\La})$
is defined as
\beq
\label{eq:master}
\Phi_K(t,z,\mathbf{\La}) = 
\sum_{(j,k) < (j',k')} (\al_j,\al_{j'}) \log(t^{(j)}_k - t^{(j')}_{k'})
- \sum_{(j,k)} \sum_{s=1}^n (\al_j, \La_s) \log(t^{(j)}_k - z_s).
\eeq
Critical points of $\Phi_K$ with respect to the $t$ variables are
are defined as solutions to the
equations
\[
\sum_{(j',k') \neq (j,k)} \frac{(\al_j, \al_{j'})}{t^{(j)}_k - t^{(j')}_{k'}}
- \sum_{s=1}^n \frac{(\al_j, \La_s)}{t^{(j)}_k - z_s}
= 0, \qquad 1 \leq j \leq r, \
1 \leq k \leq m_i.
\]
The group $\Sigma_{\mathbf{m}} = \Sigma_{m_1} \times \dots \times \Sigma_{m_r}$
acts on the critical set of $\Phi_K$ by permutation of coordinates with the same upper index.

\subsubsection{Eigenvectors}
We construct the weight function $u: \mathbb{C}^m \rightarrow V[\La - \sum_{j=1}^r m_j \al_j]$.
Let $\mathbf{b} = (b_1,\dots,b_n)$, with each $b_p$ a non-negative integer and $\sum_{p=1}^n b_p = m$.
The collection of these partitions will be denoted $B$. Let $\Sigma(\mathbf{b})$ denote the set of
bijections $\sigma$ from
the set of pairs $\{(p,s): 1\leq p \leq n, \ 1\leq s \leq b_p \}$
to the set of variables $\{t^{(1)}_1, \dots, t^{(1)}_{m_1}, \dots, t^{(1)}_1,\dots, t^{(r)}_{m_r} \}$.
Let $c(t^{(j)}_k) = j$ be the color function, and set
$c_\sigma((p,s)) = c(\sigma((p,s)))$.

For each $\mathbf{b} \in B$ and $\sigma \in \Sigma(\mathbf{b})$, we assign the vector
\[
f^\sigma_\mathbf{b} v = f_{c_\sigma((1,1))} \cdots f_{c_\sigma((1,b_1))} v_1 \otimes \cdots
\otimes f_{c_\sigma((n,1))} \cdots f_{c_\sigma((n,b_n))} v_n.
\]
Different $\sigma$ may give the same $f^\sigma_\mathbf{b}$. To $\mathbf{b}$ and $\sigma$,
we also assign the rational function
\[
u^\sigma_\mathbf{b} = u^\sigma_{\mathbf{b}, 1}(z_1) u^\sigma_{\mathbf{b}, 2}(z_2)
\dots u^\sigma_{\mathbf{b}, n}(z_n),
\]
where
\beq
\label{eq:omegafunction}
u^\sigma_{\mathbf{b}, p}(x) = \frac{1}{((\sigma(p,1)) - \sigma(p,2))(\sigma(p,2) - \sigma(p,3))
\cdots (\sigma(p,b_p -1) - \sigma(p,b_p))(\sigma(p,b_p) - x)}.
\eeq
Then we have
\beq
\label{eq:ratomega}
u(t,z) = \sum_{b\in B} \sum_{\sigma \in \Sigma(\mathbf{b})} u^\sigma_\mathbf{b} f^\sigma_\mathbf{b} v.
\eeq

\begin{thm}
\label{thm:ratbethe}
Let $t_{cr}$ be an isolated critical point of $\Phi_K(\, \cdot \, ,z, \mathbf{\La}
)$. Then $u(t_{cr},z)$ is a
well defined vector in $\sing V[\La - \mathbf{m}_\al]$ \cite{MV}.
This vector is an eigenvector of the rational
Gaudin operators $K_1(z), \dots, K_n(z)$ \cite{B, BF, RV}.
The eigenvalue of $u(t_{cr}, z)$ with respect to $K_p(z)$
is equal to $\frac{\partial}{\partial z_p} \Phi_K(t_{cr}, z, \bs{\La})$ \cite{RV}.
\end{thm}

\begin{thm}
\label{thm nonzero}
 \cite{V4}\
 For  an isolated critical point $t_{cr}$ of $\Phi_K(\, \cdot \, ,z, \mathbf{\La})$,
 the vector $u(t_{cr},z)$ is nonzero.
\end{thm}

For $\g = sl_{r+1}$ the fact that $u(t_{cr},z)$ is nonzero is proved in \cite{MTV}.

\subsubsection{Norms of eigenvectors}
For $t_{cr}$ a critical point of $\Phi_K(\, \cdot \, ,z, \mathbf{\La}, \mathbf{m})$,
let
\[
{\rm Hess}_t\,\Phi_K(t_{cr}, z, \mathbf{\La})
= {\rm det}\left(\frac{\partial^2 \Phi_K}{\partial t^{(j)}_k \partial t^{(j')}_{k'}}\right) (t_{cr})
\]
be the Hessian of $\Phi_K$.

\begin{thm}
\cite{V3}
\label{thm:ratbethenorm}
Let $t_{cr}$ be an isolated critical point of $\Phi_K(\, \cdot \, , z, \mathbf{\La}
)$.
Then
\[
S(u(t_{cr},z), u(t_{cr},z)) = {\rm Hess}_t\,\Phi_K(t_{cr}, z, \mathbf{\La}
)
\]
where $S$ is the tensor Shapovalov form on $V$.
\end{thm}

Theorem \ref{thm:ratbethenorm} was proved in \cite{MV} for $\g = sl_{r+1}$.

\begin{thm}
\cite{V3}
\label{thm:ratortho}
Let $t_{cr}$, $t'_{cr}$ be isolated critical points of $\Phi_K(\, \cdot \,, z, \mathbf{\La}
)$
lying in different orbits of $\Sigma_{\mathbf{m}}$. Then
\[
S(u(t_{cr},z), u(t'_{cr},z)) = 0.
\]
\end{thm}

\subsection{Trigonometric Gaudin Model}
Let $V = V_1 \otimes \dots \otimes V_n$, where $V_p$ are irreducible highest weight
$\g$-modules of highest weight $\La_p$ with highest weight vectors $v_p$.
Set $\mathbf{\La}  = (\La_1,\dots, \La_n)$ and $\La = \sum_p \La_p$,
$\mathbf{m} = (m_1,\dots, m_r)$, $m = \sum_j m_j$ and
$\mathbf{m}_\al = \sum_{j=1}^r m_j \al_j$ as above.
Then
the Bethe ansatz provides simultaneous eigenvectors to
the operators $\mc{K}_p(z_1, \dots, z_n, \xi)$ in $V[\La - \mathbf{m}_\al]$.

\subsubsection{Master function}
In this case we write the master function
\begin{equation}
\label{eq:trigmaster}
\Phi_{\mc{K}}(t,z,\mathbf{\La}
, \mu) = \Phi_K(t,z, \mathbf{\La})
- \sum_{(j,k)} (\al_j, \mu) \log(t^{(j)}_k)
\end{equation}
where $\Phi_K$ is given by equation (\ref{eq:master}).
The function $\Phi_{\mc{K}}$ has critical points determined by the equations
\[
\sum_{(j',k') \neq (j,k)} \frac{(\al_j, \al_{j'})}{t^{(j)}_k - t^{(j')}_{k'}}
-\sum_{p=1}^n \frac{(\al_j, \La_p)}{t^{(j)}_k - z_p}
- \frac{(\al_j,\mu)}{t^{(j)}_k} = 0 \qquad 1 \leq j \leq r, \
1 \leq k \leq m_j.
\]

\subsubsection{Eigenvectors}

\begin{thm}
\label{thm:trigbethe}
Let $t_{cr}$ be an isolated critical point of $\Phi_{\mc{K}}(\, \cdot \, ,z, \mathbf{\La}
, \xi - \rho - \frac{1}{2}(\La - \mathbf{m}_\al))$.
Then
$u(t_{cr},z) \in V[\La - \mathbf{m}_\al]$ is an eigenvector of the
trigonometric Gaudin operator $\mc{K}_p(z_1, \dots, z_n, \xi)$ with eigenvalue
equal to
\[
z_p \frac{\partial}{\partial z_p}
\Phi_{\mc{K}}(t_{cr}, z, \bs{\La}, \xi - \rho - \frac{1}{2}(\La - \mathbf{m}_\al)) +
\frac{1}{2}(\La_p, \La_p + 2\rho)
\]
for $p= 1, \dots, n$.
\end{thm}
\begin{proof}
We relate the construction under question to the Bethe ansatz for
the rational Gaudin operators $K_0(0, z_1, \dots, z_n), \dots,
K_n(0, z_1, \dots, z_n)$ on the space
$\sing M_\mu \otimes V[\mu + \La - \mathbf{m}_\al]$.

We construct the weight function
$u_{\mc{K}}: \C^m \arr M_\mu \otimes V[\mu + \La - \mathbf{m}_\al]$.
Let $\mathbf{b}_{\mc{K}} = (b_0, b_1, \dots, b_n)$, where $\sum_{p=0}^n b_p = m$,
and $B_{\mc{K}}$ denote the set of these partitions. Let $\Sigma(\mathbf{b}_{\mc{K}})$
be the set of bijections $\sigma$ from $\{(p,s): 0 \leq p \leq n, \ 1\leq s \leq b_p \}$
to $\{t^{(1)}_1, \dots, t^{(1)}_{m_1}, \dots, t^{(1)}_1,\dots, t^{(r)}_{m_r} \}$.
Let $c(t^{(j)}_k) = j$, and $c_\sigma((p,s)) = c(\sigma((p,s)))$.

For each $\mathbf{b}_{\mc{K}} \in B_{\mc{K}}$ and $\sigma \in \Sigma(\mathbf{b}_{\mc{K}})$, we assign the vector
\[
f^\sigma_{\mathbf{b}_{\mc{K}}} v_{\mc{K}} = f_{c_\sigma((0,1))} \cdots f_{c_\sigma((0,b_0))} \bs{1}_\mu \otimes \cdots
\otimes f_{c_\sigma((n,1))} \cdots f_{c_\sigma((n,b_n))} v_n.
\]
and the rational function
\[
u^\sigma_{\mathbf{b}_{\mc{K}}} = u^\sigma_{\mathbf{b}_{\mc{K}}, 0}(0) u^\sigma_{\mathbf{b}_{\mc{K}}, 1}(z_1)
\dots u^\sigma_{\mathbf{b}_{\mc{K}}, n}(z_n),
\]
where $u^\sigma_{\mathbf{b}_{\mc{K}}, p}(x)$ is as in equation (\ref{eq:omegafunction}).
Then
\[
u_{\mc{K}}(t,z) = \sum_{b_{\mc{K}} \in B_{\mc{K}}} \sum_{\sigma \in \Sigma(\mathbf{b}_{\mc{K}})}
u^\sigma_{\mathbf{b}_{\mc{K}}} f^\sigma_{\mathbf{b}_{\mc{K}}} v_{\mc{K}}.
\]

By Theorem \ref{thm:ratbethe}, if $t_{cr}$ is a critical point of $\Phi_{\mc{K}}(\, \cdot \,, z,
\mathbf{\La} , \mu)$, then $u_{\mc{K}}(t_{cr},z)$ belongs to
$\sing M_\mu \otimes V[\mu + \La - \mathbf{m}_\al]$ and is an eigenvector to
$K_0(0, z_1, \dots, z_n)$ with eigenvalue $\sum_{(j,k)} \frac{(\al_j, \mu)}{t^{(j)}_k}$
and to $K_p(0,z_1, \dots, z_n)$ 
for $p=1, \dots, n$ with eigenvalue $\frac{\partial}{\partial z_p} \Phi_{\mc{K}}(t_{cr}, z, \bs{\La}, \mu)$.

The singular vector has the form
\[
u_{\mc{K}}(t_{cr},z) = \bs{1}_\mu \otimes u(t_{cr}, z) + \sum_{j>0} F_j \bs{1}_\mu \otimes u_j
\]
for $u(t_{cr},z)$ as defined in equation (\ref{eq:ratomega}), since
$\bs{1}_\mu \otimes u(t_{cr}, z)$ is the sum of the terms of $u_{\mc{K}}(t_{cr},z)$ where $b_0 = 0$.
By Corollary \ref{cor:singularandxi}, $u(t_{cr}, z)$ is an eigenfunction
of the operators
$\mc{K}_p\left(z_1, \dots, z_n, \mu + \rho + \frac{1}{2}(\La - \mathbf{m}_\al) \right)$
for $p = 1, \dots, n$ with eigenvalue $z_p \frac{\partial}{\partial z_p}
\Phi_{\mc{K}}(t_{cr}, z, \bs{\La}, \mu) + \frac{1}{2}(\La_p, \La_p + 2\rho)$.
We let $\mu = \xi - \rho - \frac{1}{2}(\La - \mathbf{m}_\al)$ for the theorem.
\end{proof}

\begin{prop}
\label{prop:trignonzero}
Let $\xi - \rho - \frac{1}{2}(\La - \mathbf{m}_\al)$ satisfy (\ref{eqn:shapgeneric}).
For $t_{cr}$ an isolated critical point of $\Phi_{\mc{K}}(\, \cdot \, ,z, \mathbf{\La},
\xi - \rho - \frac{1}{2}(\La - \mathbf{m}_\al))$, the vector
$u(t_{cr}, z)$ is nonzero.
\end{prop}
\begin{proof}
By Theorem \ref{thm nonzero}, $u_{\mc{K}}(t_{cr}, z)$ is nonzero.
For $\mu$ satisfying (\ref{eqn:shapgeneric}), the singular vector $u_{\mc{K}}(t_{cr},z)$
equals $\Xi(\mu)(\bs{1}_\mu \otimes u(t_{cr}, z))$, so $u(t_{cr}, z)$
is nonzero.
\end{proof}

\subsubsection{Norms of eigenvectors}
For $t_{cr}$ a critical point of $\Phi_{\mc{K}}(\, \cdot \,, z, \mathbf{\La}
, \xi - \rho - \frac{1}{2}(\La - \mathbf{m}_\al))$, let ${\rm Hess}_t\,\Phi_{\mc{K}}(t_{cr})$ denote
the Hessian of $\Phi_{\mc{K}}$ with respect to the $t$ variables.

\begin{thm}
\label{thm:trigbetheprod}
For
$t_{cr}$ an isolated critical
point of $\Phi_{\mc{K}}(\, \cdot \, , z, \mathbf{\La},
, \xi - \rho - \frac{1}{2}(\La - \mathbf{m}_\al))$,
\[
\left\langle u(t_{cr},z), u(t_{cr},z) \right\rangle_\xi
= {\rm Hess}_t\,\Phi_{\mc{K}}(t_{cr} ,z, \mathbf{\La},
\xi - \rho - \frac{1}{2}(\La - \mathbf{m}_\al)).
\]
\end{thm}
\begin{proof}
Let $\mu = \xi - \rho - \frac{1}{2}(\La - \mathbf{m}_\al)$.
By Theorem \ref{thm:ratbethenorm}, we have that
\[
{\rm Hess}_t\,\Phi_{\mc{K}}(t_{cr} ,z, \mathbf{\La},\mu) =
S(u_{\mc{K}}(t_{cr},z), u_{\mc{K}}(t_{cr},z))
\]
for $S$ the tensor Shapovalov form on $M_\mu \otimes V$.
By Corollary \ref{cor:singtrigprod},
\[
\left\langle u(t_{cr},z), u(t_{cr},z) \right\rangle_\xi =
S(u_{\mc{K}}(t_{cr},z), u_{\mc{K}}(t_{cr},z)).
\]
\end{proof}

\begin{thm}
\label{thm:trigortho}
Let $t_{cr}$, $t'_{cr}$ be isolated critical
points of $\Phi_{\mc{K}}(\, \cdot \, , z, \mathbf{\La},
\xi - \rho - \frac{1}{2}(\La - \mathbf{m}_\al))$
lying in different orbits of $\Sigma_{\mathbf{m}}$.
Then
\[
\left\langle u(t_{cr},z), u(t'_{cr},z) \right\rangle_\xi = 0
\]
\end{thm}
\begin{proof}
By Corollary \ref{cor:singtrigprod}
\[
\left\langle u(t_{cr},z), u(t'_{cr},z) \right\rangle_\xi =
S(u_{\mc{K}}(t_{cr},z), u_{\mc{K}}(t'_{cr},z)),
\]
and by Theorem \ref{thm:ratortho} this is zero.
\end{proof}

\subsection{Trigonometric KZB operators}
Let $V = V_1 \otimes \dots \otimes V_n$, where $V_p$ are irreducible highest weight
$\g$-modules of highest weight $\La_p$ such that
$V$ has a non-trivial zero weight subspace $V[0]$. Let $v_p$ denote the highest weight vector
of $V_p$. Set $\mathbf{\La}  = (\La_1,\dots, \La_n)$ and $\La = \sum_p \La_p$.
For $\xi\in \h^*$,
the Bethe ansatz provides
simultaneous eigenvectors to the operators $H_p(z_1, \dots, z_n)$ in $E(\xi)$.
In this case, $\bs{m} = (m_1, \dots, m_r)$ is determined by $\bs{m}_{\al} = \La$, since
$E(\xi)$ has values in $V[0]$.

\subsubsection{Master function}
Let
$Z_s = e^{-2\pi i z_s}$
and $Z = (Z_1, \dots, Z_n)$.
We write the master function
\[
\Phi_H(t,z,\bs{\La},\mu)
= \Phi_{\mc{K}}(t, Z, \bs{\La}, \mu)
\]
with $\Phi_{\mc{K}}$ as in (\ref{eq:trigmaster}). Critical points of $\Phi_H$ with respect to
$t$ are determined by the equations
\[
\sum_{(j',k') \neq (j,k)} \frac{(\al_j, \al_{j'})}{t^{(j)}_k - t^{(j')}_{k'}}
-\sum_{p=1}^n \frac{(\al_j, \La_p)}{t^{(j)}_k - Z_p}
- \frac{(\al_j,\mu)}{t^{(j)}_k} = 0 \qquad 1 \leq j \leq r, \
1 \leq k \leq m_j.
\]

\subsubsection{Eigenfunctions}
\begin{thm}
Let $t_{cr}$ be an isolated critical point of $\Phi_H(\, \cdot \, , z, \bs{\La}, \xi - \rho)$
Then for $u(t_{cr},Z) \in V[0]$ given by (\ref{eq:ratomega}),
$\psi^\xi_{u(t_{cr},Z)}(\la)$ 
is a eigenfunction of $H_0$ with eigenvalue $\pi i(\xi,\xi)$
and of $H_p(z, \la)$ for $p =1, \dots, n$ with eigenvalue
$- \frac{1}{2\pi i} \frac{\partial}{\partial z_p} \Phi_H(t_{cr}, z, \bs{\La}, \xi - \rho)
+ (\La_p, \La_p + 2\rho)$.
\end{thm}
\begin{proof}
By Theorem \ref{thm:trigbethe}, $u(t_{cr},Z) \in V[0]$ is eigenvector of the
trigonometric Gaudin operators $\mc{K}_p(Z_1, \dots, Z_n, \xi)$ for $p=1,\dots, n$ with
eigenvalue $Z_p \frac{\partial}{\partial Z_p} \Phi_{\mc{K}}(t_{cr}, Z, \bs{\La}, \xi - \rho) +
(\La_p, \La_p + 2\rho)$.
Lemma \ref{lem:trigeigenfunct} implies that $\psi^\xi_{u(t_{cr},Z)}$ is a
eigenfunction of $H_p(z)$ for $p=1,\dots, n$ with the same eigenvalue.
\end{proof}

\begin{prop}
For $\xi - \rho \in \h^*$ satisfying (\ref{eqn:shapgeneric}), and
$t_{cr}$ an isolated critical point of $\Phi_H(\, \cdot \, , z, \bs{\La}, \xi - \rho)$,
the function $\psi^\xi_{u(t_{cr},Z)}(\la)$ is nonzero.
\end{prop}
\begin{proof}
By Proposition \ref{prop:trignonzero}, $u(t_{cr},z)$ is nonzero,
so $\psi^\xi_{u(t_{cr},Z)}(\la)$ is nonzero.
\end{proof}

\subsubsection{Norms of eigenfunctions}
For $t_{cr}$ a critical point of $\Phi_H(\, \cdot \, , z, \bs{\La}, \xi-\rho)$,
let ${\rm Hess}_t\,\Phi_H(t_{cr})$ denote
the Hessian of $\Phi_H$ with respect to the $t$ variables.

\begin{thm}
\label{thm:kzbnorm}
Let  $t_{cr}$ be an isolated critical point of $\Phi_H(\, \cdot\, , z, \mathbf{\La},
\xi -\rho)$. 
Then
\[
\left\langle \psi^\xi_{u(t_{cr},Z)}, \psi^\xi_{u(t_{cr},Z)} \right\rangle
= 
{\rm Hess}_t \, \Phi_H(t_{cr} , z,  \mathbf{\La}, \xi -\rho).
\]
\end{thm}
\begin{proof}
By Proposition \ref{prop:eigentrigprod},
\[
\left\langle \psi^\xi_{u(t_{cr},Z)}, \psi^\xi_{u(t_{cr},Z)} \right\rangle
= \left\langle u(t_{cr},Z), u(t_{cr},Z) \right\rangle_\xi.
\]
Since $t_{cr}$ is a isolated critical point of
$\Phi_{\mc{K}}(\, \cdot \, , Z, \bs{\La}, \xi - \rho)$,
by Theorem \ref{thm:trigbetheprod}, we have
\[
\left\langle u(t_{cr},Z), u(t_{cr},Z) \right\rangle_\xi =
{\rm Hess}_t \,\Phi_{\mc{K}}(t_{cr} ,Z, \mathbf{\La}, \xi - \rho).
\]
\end{proof}

\begin{thm}
Let
$t_{cr}$, $t'_{cr}$ be isolated critical points of
$\Phi_H(\, \cdot\, , z, \mathbf{\La}, \xi -\rho)$ 
lying in different $\Sigma_{\bs{\La}}$ orbits. Then
\[
\left\langle \psi^\xi_{u(t_{cr},Z)}, \psi^\xi_{u(t'_{cr},Z)} \right\rangle
= 0.
\]
\end{thm}
\begin{proof}
By Proposition \ref{prop:eigentrigprod},
\[
\left\langle \psi^\xi_{u(t_{cr},Z)}, \psi^\xi_{u(t'_{cr},Z)} \right\rangle
= \left\langle u(t_{cr},Z), u(t'_{cr},Z) \right\rangle_\xi,
\]
and by Theorem \ref{thm:trigortho}, this is zero.
\end{proof}

\section{Weyl group}
\label{sect:weylgroup}
For $\al \in \Dl$, we have the reflection $s_\al$ of $\h$, defined by
$s_\al(\la) = \la - (\la, \al^{\vee})\al$. The Weyl group $W$
associated to $\g$ is the group of transformations of $\h$ generated by such $s_\al$.
The simple reflections $s_j = s_{\al_j}$ generate $W$.
The Weyl group acts on $V[0]$ so it acts on
$V[0]$-valued functions of $\h$ by $(w\psi)(\la) = w(\psi(w^{-1}\la))$.

\begin{lem}
\cite{FW} The operators $H_0(z,\tau)$, $H_1(z,\tau), \dots, H_n(z,\tau)$ are Weyl invariant.
\end{lem}

\begin{cor}
Let $\psi$ be an eigenfunction of $H_p$ for $p$ any of $0, 1, \dots, n$.
Then for $w \in W$, $w \psi$ is an eigenfunction of $H_p$ with the same eigenvalue.
\end{cor}

\subsection{Scattering matrices}

Following \cite{FV2}, we define the maps $T_w(\xi): V[0] \arr V[0]$, rational
in the variable $\xi \in \h^*$ and associated to $w \in W$. For $s_j$ a simple
reflection, we define
\begin{equation}
\label{eqn:dynamicalweylsum}
T_{s_j}(\xi) = \sum_{\ell=0}^\infty (-1)^{\ell}
\left(\frac{1}{\ell!}\right)^2 \frac{(\xi, \al_j^\vee)}{(\xi, \al_j^\vee) - \ell} f_j^\ell e_j^\ell.
\end{equation}
For $w \in W$ with decomposition
$w = s_{j_m} \cdots s_{j_2} s_{j_1}$ by simple reflections, we define
\[
T_w(\xi) = T_{s_{j_m}}(s_{j_{m-1}} \cdots s_{j_2} s_{j_1} \xi) \cdots
T_{s_{j_2}}(s_{j_1} \xi) T_{s_{j_1}}(\xi).
\]

\begin{prop}
\cite{FV2} The map $T_w(\xi)$ does not depend on choice of decomposition
$w = s_{j_m} \cdots s_{j_2} s_{j_1}$.
\end{prop}

\begin{cor}
For $w_1, w_2 \in W$, the composition property
\[
T_{w_2 w_1}(\xi) = T_{w_2}(w_1 \xi)T_{w_1}(\xi).
\]
holds.
\end{cor}

We note that this map $T_w(\xi)$ is identical to the dynamical Weyl
group element $A_w(\xi - \rho)$ acting on $V[0]$, see \cite{TV}, \cite{EV1}, \cite{STV}.

Let $sl_2(j)$ denote the
subalgebra of $\g$ generated by $e_{j}$ and $f_{j}$
and let $V^{(j)}_k$ be a $2k+1$-dimensional irreducible $sl_2(j)$ submodule
of $V$.

\begin{prop}
For $u \in V^{(j)}_k[0]$, the explicit formula
\[
T_{s_j}(\xi)u = \frac{(1+ \xi_j)(2+ \xi_j)\cdots(k+ \xi_j)}
{(1 - \xi_j)(2-\xi_j) \cdots (k-\xi_j)} u
\]
holds, where $\xi_j$ denotes $(\xi, \al_j^\vee)$.
\end{prop}
\begin{proof}
Since $e_j^{k+1}$ applied to $u \in V^{(j)}_k[0]$ is zero, we have
\[
T_{s_j}(\xi)u = \sum_{\ell=0}^k (-1)^{\ell}
\left(\frac{1}{\ell!}\right)^2 \frac{(\xi, \al_j^\vee)}{(\xi, \al_j^\vee) - \ell} f_j^\ell e_j^\ell u.
\]
Each $f_j^\ell e_j^\ell u$ is equal to $\frac{(k + \ell)!}{(k - \ell)!} u$,
so we have
\[
T_{s_j}(\xi)u = \sum_{\ell=0}^k (-1)^\ell \left(\frac{1}{\ell!}\right)^2
\frac{(\xi, \al_j^\vee)(k+\ell)!}{((\xi,\al_j^\vee)-\ell)(k-\ell)!} u.
\]
Thus, $T_{s_j}(\xi)u $ is a degree $k$ rational function of $\xi$
with simple poles at $(\xi, \al_j^\vee)$ equal to $1, 2, \dots, k$.
The residue at $\ell$ for $\ell \in \{1, 2, \dots, k\}$ is given by
\[
{\rm Res}_{(\xi, \al_j^\vee) = \ell} T_{s_j}(\xi)u =
(-1)^\ell \frac{(k+\ell)!}{\ell!(\ell-1)!(k-\ell)!}.
\]
This proves the proposition up to a constant multiple. The constant is
fixed by $T_{s_j}(0) = 1$.
\end{proof}

\begin{lem}
\cite{FV2}
Let $\xi$ satisfy (\ref{eqn:functiongeneric}). Then for all $w \in W$, the map
$\psi \mapsto w\psi$ is an isomorphism from $E(\xi)$ to $E(w\xi)$.
\end{lem}

\begin{thm}
\label{thm:scattering}
\cite{FV2}
The maps $T_w(\xi)$ satisfy the relation
\[
w \psi^\xi_u = \psi^{w\xi}_{T_w(\xi)u}
\]
for $u \in V[0]_{\xi - \rho}$.
\end{thm}

Let $T_w(\xi)^*$ denote the adjoint operator to $T_w(\xi)$ with respect
to the Shapovalov form.

\begin{prop}
\label{prop:Asymm}
For a simple reflection $s_j$, we have that $T_{s_j}(\xi)^* = T_{s_j}(\xi)$.
\end{prop}
\begin{proof}
This follows from the fact that $f_j^\ell e_j^\ell$ is self-adjoint with
respect to the Shapovalov form.
\end{proof}

\begin{cor}
\label{cor:dynamicaladjoint}
For $\xi \in \h$ and $w = s_{j_m} \cdots s_{j_1}$, we have
\[
T_w(\xi)^* = T_{s_{j_1}}(\xi) T_{s_{j_2}}(s_{j_1} \xi) \cdots
T_{s_{j_m}}(s_{j_{m-1}} \cdots s_{j_1} \xi).
\]
\end{cor}

\begin{lem}
\label{lem:commutesimple}
For $s_j$ a simple reflection and $w_0$ the longest element of $W$, the relation
\[
w_0 T_{s_j}(\xi) w_0^{-1} = T_{w_0 s_j w_0^{-1}}(w_0 s_j \xi).
\]
holds.
\end{lem}
\begin{proof}
We note the relation $w (e_{-\al}^l e_{\al}^l) w^{-1} =
e_{-w\al}^l e_{w\al}^l$ \cite{TV}. The simple root
$- w_0 \al_j$ equals $w_0 s_j \al_j$, so we have
$w_0 (f_{\al_j}^l e_{\al_j}^l) w_0^{-1} = e_{w_0 s_j \al_j}^l
f_{w_0 s_j \al_j}^l$. For $u \in V[0]$, the equation
$e_\al^l f_\al^l u = f_\al^l e_\al^l u$ holds, so we have
\[
w_0 T_{s_j} w_0^{-1} = \sum_{\ell=0}^\infty (-1)^{\ell}
\left(\frac{1}{\ell!}\right)^2 \frac{(\xi, \al_j^\vee)}{(\xi, \al_j^\vee) - \ell}
f_{w_0 s_j \al_j}^\ell e_{w_0 s_j \al_j}^\ell.
\]
Since $(\xi, \al_j^\vee)$ equals $(w_0 s_j \xi, w_0 s_j \al_j^\vee)$
we have
\[
w_0 T_{s_j} w_0^{-1} = T_{s_{w_0 s_j \al_j}}(w_0 s_j \xi).
\]
The simple reflection $s_{w_0 s_j \al_j}$ is equal to $w_0 s_j w_0^{-1}$
which gives the lemma.
\end{proof}

\begin{prop}
\label{prop:commuteadjoint}
For $w \in W$ and $w_0$ the longest element of $W$, the relation
\[
w_0 T_{w}(\xi)^* w_0^{-1} = T_{w_0 w^{-1} w_0^{-1}}(w_0 w \xi).
\]
holds.
\end{prop}
\begin{proof}
By Corollary \ref{cor:dynamicaladjoint}, the operator $T_w(\xi)^*$
can be written as
\[
T_w(\xi)^* =
T_{s_{i_1}}(\xi) T_{s_{i_2}}(s_{i_1} \xi) \cdots T_{s_{i_k}}(s_{i_{k-1}} \cdots s_{i_1} \xi)
\]
for $w = s_{i_k} \cdots s_{i_2} s_{i_1}$. Applying Lemma \ref{lem:commutesimple}
successively to the simple reflections we obtain
\[
w_0 T_w(\xi)^* w_0^{-1} = T_{w_0 s_{i_1} w_0^{-1}}(w_0 s_{i_1} \xi) T_{w_0 s_{i_2} w_0^{-1}}
(w_0 s_{i_2} s_{i_1} \xi) \cdots T_{w_0 s_{i_k} w_0^{-1}}(w_0 w \xi).
\]
The composition property is applied to give the proposition.
\end{proof}

\begin{thm}
\label{prop:QandT}
\cite{EV1}
As elements of $End_{\C}V[0] \otimes \C(\h^*)$, we have
\[
\mc{Q}(\xi) = w_0 T_{w_0}(\xi).
\]
\end{thm}

\subsection{Scalar products}

\begin{thm}
\label{thm:weyltrigprod}
For $u, v \in V[0]_{\xi -\rho}$ and $w \in W$, we have
\[
\langle T_w(\xi) u, T_w(\xi)v \rangle_{w \xi} =
\langle u, v \rangle_\xi.
\]
\end{thm}
\begin{proof}
By Theorem \ref{prop:QandT}, we have that
\[
\langle u, v \rangle_\xi
= S(u, w_0 T_{w_0}(\xi) v)
\]
and that
\[
\langle T_w(\xi) u, T_w(\xi)v \rangle_{w \xi}
= S( T_w(\xi) u, w_0 T_{w_0}(w \xi) T_w(\xi) v).
\]
The definition of the adjoint operator $T_w(\xi)^*$ gives
\[
S( T_w(\xi) u, w_0 T_{w_0}(w \xi) T_w(\xi) v)
= S( u, T_w(\xi)^* w_0 T_{w_0}(w \xi) T_w(\xi)v).
\]
By Proposition \ref{prop:commuteadjoint}, we have
\[
S( u, T_w(\xi)^* w_0 T_{w_0}(w \xi) T_w(\xi)v)
= S( u, w_0 T_{w_0 w^{-1} w_0^{-1}}(w_0 w \xi) T_{w_0}(w \xi) T_w(\xi) v).
\]
The cocycle condition gives
\[
S( u, w_0 T_{w_0 w^{-1} w_0^{-1}}(w_0 w \xi) T_{w_0}(w \xi) T_w(\xi) v)
= S( u, w_0 T_{w_0}(\xi) v)
\]
which completes the proof.
\end{proof}

\begin{cor}
\label{cor:weylprod}
For $\xi \in \h^*$, let $u$ and $v$ belong to $V[0]_{\xi - \rho}$. Then
\[
\langle w \psi^\xi_u, w\psi^\xi_v \rangle = \langle \psi^\xi_u, \psi^\xi_v \rangle.
\]
holds.
\end{cor}
\begin{proof}
The left hand side equals
$\langle \psi^{w\xi}_{T_w(\xi)u}, \psi^{w\xi}_{T_w(\xi)v} \rangle$
by Theorem \ref{thm:scattering}.
By Proposition \ref{prop:eigentrigprod}, the corollary is equivalent to
\[
\langle T_w(\xi) u, T_w(\xi)v \rangle_{w\xi} = \langle u, v \rangle_\xi
\]
which is the statement of Theorem \ref{thm:weyltrigprod}
\end{proof}

\begin{prop}
\label{prop:crossterms}
Let $\xi \in \h^*$ and $w \in W$ be such that
$w \xi = \xi + \beta$ for $\beta \in Q$ nonzero.
For $u, v \in V[0]_{\xi-\rho}$,
\[
\langle \psi^\xi_u, w \psi^\xi_v \rangle = 0.
\]
holds.
\end{prop}
\begin{proof}
The proposition follows from Proposition \ref{prop:crosstermsorth},
because of the relation
\[
\langle \psi^\xi_u, w \psi^\xi_v \rangle =
\langle \psi^\xi_u, \psi^{\xi+\beta}_{T_w(\xi)v} \rangle.
\]
\end{proof}

For $u \in V[0]_{\xi -\rho}$, we define
\[
\psi^{W\xi}_u = \sum_{w \in W} (-1)^{l(w)} w \psi^\xi_u
\]
where $l(w)$ denotes the length of $w$.
For $\xi$ integral, each $w \xi$ differs from $\xi$ by an element of
the root lattice, so $\langle \psi^\xi_u, \psi^{w \xi}_v \rangle$ is
well-defined.

\begin{cor}
\label{cor:weylproduct}
Let $\xi \in \h^*$ be integral,
with $w\xi \neq \xi$ for every $w \in W$.
For $u, v \in V[0]_{\xi -\rho}$, the equation
\[
\langle \psi^{W\xi}_u, \psi^{W\xi}_v\rangle =
|W| \langle \psi^\xi_u, \psi^\xi_v \rangle.
\]
holds.
\end{cor}
\begin{proof}
By Proposition \ref{prop:crossterms}, for $w_1 \neq w_2$,
$\langle w_1 \psi^\xi_u, w_2 \psi^\xi_v \rangle$ equals zero.
By Corollary \ref{cor:weylprod}, the terms
$\langle w \psi^\xi_u, w \psi^\xi_v \rangle$
for each $w \in W$ are all equal.
\end{proof}

\section{Jack polynomials}
\label{sect:jack}
In this section, we set $\g = sl_{r+1}$, and let the representation $V$ be $S^{k(r+1)}\C^{r+1}$.
Then $V[0]$ is one-dimensional. We set $z_1 = 0$.

Denote the fundamental weights of $\g$ by $\omega_1, \dots, \omega_r$. Let $P_+$ denote the integral
dominant weights, the linear combinations of $\omega_1, \dots, \omega_r$ with non-negative integral
coeffecients.

\begin{prop}
For $\mu \in \h$ with $\mu - k \rho \in P_+$, the equality $V[0]_\mu = V[0]$ holds.
\end{prop}
\begin{proof}
For
\[
\Xi(\mu)(\bs{1}_\mu \otimes u) = \sum_{\ell,m \geq 0} (S_\mu^{-1})_{\ell m} F_\ell \bs{1}_\mu \otimes \omega(F_m)\otimes u
\]
to be defined, it is necessary and sufficient that for each $F$ such that $F \, \bs{1}_\mu$ is in the
kernel of the Shapovalov form, the expression $\omega(F)u$ is zero. The kernel is generated by the
highest weight vectors of the subrepresentations $M_{s_j \cdot \mu}$ of $M_\mu$ for simple reflections
$s_j$, which are given by $f_j^{(\mu + \rho, \al_j^\vee)} \bs{1}_\mu$. Since $(\mu + \rho, \al_j^\vee)$
is at least $k+1$, it is sufficient to verify that $e_j^{k+1} u$ is zero for $u \in V[0]$.
This is true since the weight spaces $V[(k+1)\al_j]$ are empty.
\end{proof}

We use the variables $X_{\omega_j} = e^{-2 \pi i \omega_j(\la)}$ for $\omega_j$ a fundamental weight.
For $k$ a fixed non-negative integer, the Jack polynomials $P^{(k)}_\nu$ are a family
of Weyl-invariant polynomials in $X^{\pm 1}_{\omega_1}, \dots, X^{\pm 1}_{\omega_r}$
parametrized by dominant weights $\nu$.
They are characterized by their form $P^{(k)}_\nu = X_{-\nu} + \sum_{\beta \in Q_+} a_\beta X_{-\nu + \beta}$
for coefficients $a_\beta \in \C$, and their orthogonality with respect to the inner product
\[
\langle \phi_1(\la), \phi_2(\la) \rangle_k = \frac{1}{|W|} {\rm const}_{\{X_{\omega_j}\}} \phi_1(\la) \phi_2(-\la)
\prod_{\al \in \Dl}(1- X_\al)^{k+1}
\]
where ${\rm const}_{\{X_{\omega_j}\}}$ is the constant term with respect to the $X_{\omega_j}$ variables.

Let $\Pi$ denote the product $X_{-\rho} \prod_{\al \in \Dl_+} (1 - X_\al)$.

\begin{prop}
\cite{FV2}
\label{prop:jackeigen}
For $u \in V[0]$ and $\xi \in \h$ such that $\xi - (k+1) \rho$ belongs to $P_+$,
the antisymmetrized eigenfunction $\psi^{W\xi}_u$ has the form
\[
\psi^{W\xi}_u = \Pi^{k+1} P^{(k)}_{\xi - (k+1)\rho} u.
\]
\end{prop}

\begin{prop}
\label{prop:eigenjacknorm}
Let $\xi, \nu \in \h^*$ be such that $\xi - (k+1)\rho$ and $\nu - (k+1)\rho$
belong to $P_+$ with $\xi - \nu \in \Dl$. For $u, v \in V[0]$, we have the relation
\[
\langle \psi^{W \xi}_u, \psi^{W \nu}_v \rangle = |W| \langle P^{(k)}_{\xi - (k+1)\rho},
P^{(k)}_{\nu - (k+1)\rho} \rangle_k S(u,v).
\]
\end{prop}
\begin{proof}
The left hand side is defined as
\[
\langle \psi^{W \xi}_u, \psi^{W \nu}_v \rangle =
\int_{C} S(\psi^{W \xi}_u (\la), \psi^{W \nu}_v(-\la)) d\la_1 d\la_2 \cdots d\la_r,
\]
where $C$ is given by each $\la_j = (\la, \al_j)$ ranging along the interval
from $0 - i \delta$ to $1 - i \delta$.
Since $\xi - \nu$ is in $\Dl$, the integrand is a meromorphic function of the variables
$X_1, \dots, X_r$, so the product is
\[
\langle \psi^{W \xi}_u, \psi^{W \nu}_v \rangle = {\rm const}_{\{X_j\}} S(\psi^{W\xi}_u (\la),
\psi^{W\nu}_v (-\la)).
\]

By Proposition \ref{prop:jackeigen}, we have
\[
{\rm const}_{\{X_j\}} S(\psi^{W\xi}_u, \psi^{W\nu}_v (-\la)) =
{\rm const}_{\{X_j\}} S(\Pi(\la)^{k+1} P^{(k)}_{\xi - (k+1)\rho}(\la) u,
\Pi(-\la)^{k+1} P^{(k)}_{\nu - (k+1)\rho}(-\la) v).
\]
Since $\Pi(-\la) = X_\rho \prod_{\al \in \Dl_-}(1-X_{\al})$, we have
\[
\langle \psi^{W \xi}_u, \psi^{W \nu}_v \rangle =
{\rm const}_{\{X_j\}} \prod_{\al \in \Dl} (1-X_{\al}) P^{(k)}_{\xi - (k+1)\rho}(\la)
P^{(k)}_{\nu - (k+1)\rho}(-\la) S(u,v).
\]
On the other hand, we have that
\[
|W|\langle P^{(k)}_{\xi - (k+1)\rho}, P^{(k)}_{\nu - (k+1)\rho} \rangle_k S(u,v)
= {\rm const}_{\{X_{\omega_j}\}} \prod_{\al \in \Dl} (1-X_{\al}) P^{(k)}_{\xi - (k+1)\rho}(\la)
P^{(k)}_{\nu - (k+1)\rho}(-\la).
\]
Since the expression on the right has well-defined constant terms in both the $X_j$ variables
and the $X_{\omega_j}$ variables, and since each $X_j$ is a non-constant
product of integer powers of the $X_{\omega_j}$ variables, the two constant terms are equal.
\end{proof}

\subsection{Bethe ansatz}
We apply the Bethe ansatz results in this case of $\g = sl_{r+1}$ and $V= S^{k(r+1)}\C^{r+1}$.
The highest weight of $V$ is $k(r+1) \omega_1 = \sum_{j=1}^r k (r+1-j) \al_j$. Thus the $t$ variable
in $\C^{kr(r+1)/2}$ has variables $t^{(j)}_\ell$ for $1\leq j \leq r$ and $1 \leq \ell \leq k(r+1-j)$.
The master function $\Phi_H(t,k(r+1) \omega_1, \xi-\rho)$ and the weight function $u(t)$ are as defined
previously, with $z_1 =1$.

The Bethe ansatz asserts that if $t_{cr}$ is an isolated critical point of
$\Phi_H(t,k(r+1) \omega_1, \xi-\rho)$, the function
$\psi^\xi_{u(t_{cr})}$ is an eigenfunction of $H_0$.
Theorem \ref{thm:kzbnorm} gives the norm of $\psi^\xi_{u(t_{cr})}$ as
\[
\left\langle \psi^\xi_{u(t_{cr})}, \psi^\xi_{u(t_{cr})} \right\rangle
=
{\rm Hess}_t \, \Phi_H(t,k(r+1) \omega_1, \xi-\rho).
\]

\begin{prop}
If $t_{cr}$ is an isolated critical point of $\Phi_H(t, \nu + k\rho)$,
then the relation
\[
\left\langle P^{(k)}_\nu , P^{(k)}_\nu \right\rangle_k S(u(t_{cr}), u(t_{cr})) =
{\rm Hess}_t \, \Phi_H(t_{cr},k(r+1) \omega_1, \nu + k\rho).
\]
holds.
\end{prop}
\begin{proof}
By Corollary \ref{cor:weylproduct}, the Bethe function
$\psi^{\nu + (k+1)\rho}_{u(t_{cr})}$ has the same norm as
its antisymmetrization $\psi^{W(\nu +(k+1)\rho)}_{u(t_{cr})}$ divided
by $|W|$, so
Proposition \ref{prop:eigenjacknorm} relates the
norm of Jack polynomial to the norm of the Bethe function.
\end{proof}

\section{Lie algebra $sl_2$, one tensor factor}
\subsection{Norm of eigenfunction}
We use explicit formulas to show Theorem \ref{thm:kzbnorm} in the case
$\g = sl_2$ and $V= V_\La$ consists of a single irreducible factor with
$\La = k \al_1$ where $k$ is a non-negative integer. We set $z_1 = 0$, $X = X_1$,
$\xi_1 = (\xi, \al_1)$ and $\la_1 = (\la, \al_1)$.
The space $V[0]$ is one dimensional
and we have
\[
H_0 =
\frac{1}{2 \pi i} \frac{d^2}{d\la_1^2} + \frac{\pi i k(k+1)}{2 \sin^2(\pi \la_1)}.
\]

\begin{lem}
\label{lem:sl2}
Let $V=V_\La$, with $\La = k\al_1$, 
and $u \in V[0]$.
For $\xi_1$ not equal to any of $\{1, \dots, k\}$, we have
\[
\left\langle \psi^\xi_u, \psi^\xi_u \right\rangle
= 
\frac{(\xi_1 + 1) (\xi_1 + 2) \cdots (\xi_1 +k)}
{(\xi_1 -1) (\xi_1 - 2) \cdots (\xi_1 -k)}
S(u,u).
\]
\end{lem}
\begin{proof}
The formula for $\psi^\xi_u$ in this case is calculated in \cite{FV2} as
\[
\psi^\xi_u(\la) = e^{2\pi i \xi(\la)} \sum_{j = 0}^k (-1)^j
\frac{(k+j)! \Gm(\xi_1 - j)}{j! (k - j)! \Gm(\xi_1)}
\left( \frac{X}{1-X} \right)^j u.
\]
We write
\[
\psi^\xi_u(\la) = C_{k, \xi} \
e^{2 \pi i \xi(\la)} P^{(-\xi_1, \xi_1)}_k \left( \frac{1+X}{1-X} \right) u
\]
with the Jacobi polynomial \cite{S}
\[
P^{(\al, \beta)}_k(z) = \frac{\Gm(\al + k +1)}{k!\Gm(\al + \beta + k +1)}
\sum_{j=0}^k \binom{k}{j} \frac{\Gm(\al + \beta + k + j + 1)}{\Gm(\al + j + 1)} \left(\frac{z-1}{2}\right)^j,
\]
and the constant $C_{k, \xi} = \frac{1}{P_k^{(-\xi_1,\xi_1)}(1)}$.
The norm of $\psi^\xi_u$ is
\begin{align*}
\left\langle \psi^\xi_u, \psi^\xi_u \right\rangle & =
{\rm const}_X S\left( C_{k, \xi} \ P^{(-\xi_1, \xi_1)}_k\left(\frac{1+X}{1-X}\right) u,
C_{k, \xi} \ P^{(-\xi_1, \xi_1)}_k\left(\frac{1+X^{-1}}{1-X^{-1}}\right) u \right)\\
&  = C_{k, \xi}^2  \ {\rm const}_X \ P^{(-\xi_1, \xi_1)}_k \left(\frac{1+X}{1-X}\right)
P^{(-\xi_1, \xi_1)}_k \left(\frac{X+1}{X-1}\right) \ S(u,u) \\
& = \frac{P^{(-\xi_1, \xi_1)}_k(-1)}{P^{(-\xi_1, \xi_1)}_k(1)} \ S(u,u).
\end{align*}
The special values
\[
P^{(\al, \beta)}_k(1) = \binom{k + \al}{k}
\]
and
\[
P^{(\al, \beta)}_k(-1) = (-1)^k \binom{k + \beta}{k}
\]
give
\[
\left\langle \psi^\xi_u, \psi^\xi_u \right\rangle =
(-1)^k \frac{\Gm(k + \xi_1 + 1) \Gm(-\xi_1 + 1)}
{\Gm(\xi_1 + 1) \Gm(k-\xi_1 +1)}
\]
which is equivalent to the Lemma.
\end{proof}

We suppress the upper index of the $t$ variables,
and write $t = (t_1, \dots, t_k)$. The master function is
\[
\Phi_H(t, z, \La, \xi - \rho) =
\sum_{j\neq j'} 2 \log(t_j - t_{j'}) - \sum_{j=1}^k 2k \log(t_j - 1)
- \sum_{j=1}^k (\xi_1 - 1) \log(t_j).
\]

\begin{thm}
\label{thm:sl2product}
For $t_{cr}$ a critical point of $\Phi_H$,
we have that
\[
\left\langle \psi^\xi_{u(t_{cr}, Z)}, \psi^\xi_{u(t_{cr}, Z)} \right\rangle
= \frac{\left[(\xi_1 + 1)(\xi_1+2) \cdots (\xi_1+k)\right]^3 k!}
{(\xi_1-1)(\xi_1 - 2) \cdots (\xi_1 -k) (k+1) (k+2)\cdots (2k)}.
\]
\end{thm}
\begin{proof}
By applying Lemma \ref{lem:sl2}, we need only show
\[
S(u(t_{cr},Z), u(t_{cr},Z)) =
\left(\frac{(\xi_1+1)(\xi_1 +2) \cdots (\xi_1 +k)}{(k+1)(k+2) \cdots (2k)}\right)^2 (2k)!.
\]
The weight function is
\[
u(t, Z) = \sum_{\sigma \in S_k} \frac{1}{(t_{\sigma(1)} - t_{\sigma(2)})
(t_{\sigma(2)} - t_{\sigma(3)}) \cdots (t_{\sigma(k-1)} - t_{\sigma(k)})
(t_{\sigma(k)} - 1)} f^k  v_\La,
\]
which can be simplified to
\[
u(t,Z) = \left(\prod_{j=1}^k \frac{1}{t_j - 1}\right) f^k  v_\La.
\]
For $t$ a critical point of $\Phi_H$, \cite{V2} gives the formula
\[
\prod_{j=1}^k \frac{1}{t_j - 1} =
\frac{(\xi_1+1) (\xi_1 +2) \cdots (\xi_1 + k)}{(k+1)(k+2)\cdots(2k)}.
\]
Calculation shows that $S(f^k v_\La, f^k v_\La) = (2k)!$, which gives the theorem.
\end{proof}

By \cite[Equation 1.4.3]{V2}, we have
\[
{\rm Hess_t} \, (\Phi_H(t_{cr}, z, \La, \xi-\rho)) = k! \prod_{j=0}^{k-1}
\frac{(\xi_1 + 1 + j)^3}{(\xi - 1 - j)(2k-j)}
\]
to obtain the claim of Theorem \ref{thm:kzbnorm} that
\[
\left\langle \psi^\xi_{u(t_{cr},Z)}, \psi^\xi_{u(t_{cr},Z)} \right\rangle
= {\rm Hess}_t \, \Phi_H(t_{cr} , z,  \mathbf{\La}, \xi -\rho).
\]

\end{document}